\documentclass{ws-ppl}
\usepackage{booktabs}
\usepackage{cite}
\usepackage[dvipsnames]{xcolor}
\usepackage[final]{listings}
\usepackage{url}
\begin{document}
\markboth{V.~Novakovi\'{c}}{Recursive vectorized computation of the vector $p$-norm}
%
%
\catchline{}{}{}{}{}
%
%
\title{Recursive vectorized computation of the vector $p$-norm}
\author{Vedran Novakovi\'{c}}
\address{independent researcher, \url{https://orcid.org/0000-0003-2964-9674}\\
  Vankina ulica 15, HR-10020 Zagreb, Croatia\\
  e-mail address: \texttt{venovako@venovako.eu}}
\maketitle
\begin{history}
\received{(received date)}
\revised{(revised date)}
\end{history}
\begin{abstract}
  Recursive algorithms for computing the Frobenius norm of a real
  array are proposed, based on $\mathrm{hypot}$, a hypotenuse
  function.  Comparing their relative accuracy bounds with those of
  the BLAS routine $\mathtt{DNRM2}$ it is shown that the proposed
  algorithms could in many cases be significantly more accurate.  The
  scalar recursive algorithms are vectorized with the Intel's vector
  instructions to achieve performance comparable to $\mathtt{DNRM2}$,
  and are further parallelized with OpenCilk.  Some scalar algorithms
  are unconditionally bitwise reproducible, while the reproducibility
  of the vector ones depends on the vector width.  A modification of
  the proposed algorithms to compute the vector $p$-norm is also
  presented.
\end{abstract}
\keywords{Frobenius norm; AVX-512 vectorization; roundoff analysis; vector $p$-norm.}
\category{Mathematics Subject Classification (2020)}{65F35, 65Y05, 65G50}
\begin{otherinfo}
  Supplementary material, including an implementation of the proposed
  algorithms, is available in
  \url{https://github.com/venovako/VecNrmP} and
  \url{https://github.com/venovako/libpvn} repositories.
\end{otherinfo}
\newlength\fbw
\newlength\fbh
\newlength\fba
\newlength\fbt
\lstloadlanguages{C}
\lstset{language=C,extendedchars=false,numbers=left,numberstyle=\tiny,frame=lines,basicstyle=\small\ttfamily,commentstyle=\small\normalfont,columns=fullflexible,texcl=false,mathescape=true,lineskip=.25\baselineskip}
%
%
\section{Introduction}\label{s:1}
%
%
For a real $p\ge 1$, the vector $p$-norm (or $\ell^p$ norm) of an
array $\mathbf{x}$ is defined as
\begin{equation}
  \|\mathbf{x}\|_p^{}=\left(\sum_{i=1}^n|x_i^{}|^p\right)^{1/p},\qquad
  \mathbf{x}=[x_1\cdots x_n],\label{e:p}
\end{equation}
with the most common instances of $p$ in algorithms of numerical
linear algebra being $p=1$, $p=2$, and $p=\infty$, i.e.,
\begin{equation}
  \|\mathbf{x}\|_1^{}=\sum_{i=1}^n|x_i^{}|,\qquad
  \|\mathbf{x}\|_2^{}=\sqrt{\sum_{i=1}^n|x_i^{}|^2},\qquad
  \|\mathbf{x}\|_{\infty}^{}=\max_{i=1,\ldots,n}^{}|x_i^{}|,
  \label{e:12I}
\end{equation}
where the vector $2$-norm is often called Frobenius and denoted by
$\|\mathbf{x}\|_F$.  The widely used routine $\mathtt{xNRM2}$ for
computation of the Frobenius norm of a one-dimensional real array
without undue overflow, as implemented in the Reference BLAS in
Fortran\footnote{See
\url{https://github.com/Reference-LAPACK/lapack/blob/master/BLAS/SRC/dnrm2.f90}
(double precision) in the Reference LAPACK~\cite{Anderson-et-al-99}
repository, or \texttt{snrm2.f90} for the single precision version.},
is sequential and prone to the accumulation of rounding errors, and to
other numerical issues, for inputs with a large number of elements.
This work proposes an alternative algorithm, $\mathtt{xNRMF}$, that
improves the theoretical error bounds (due to its recursive nature)
and the observed accuracy on large random inputs with the moderately
varying magnitudes of the elements, while still exhibiting comparable
performance (due to vectorization), in single
($\mathtt{x}=\mathtt{S}$) and double ($\mathtt{x}=\mathtt{D}$)
precision.

It suffices to focus on reals arrays only, since
$|x_i^{}|^2=(\Re{x_i^{}})^2+(\Im{x_i^{}})^2$ for a complex $x_i$.  The
Frobenius norm of a multi-dimensional array (e.g., a matrix) can be
constructed from the norms of its lower-dimensional subarrays (e.g.,
columns), and thus only one-dimensional arrays are considered.  The
norm of a scalar $x$ is $|x|$.

Even though $\mathtt{xNRMF}$ is not the most performant stable
norm-computation routine available, one of its strengths is that it is
conceptually simple, and another one is that it can be generalized to
the $\mathtt{xNRMP}$ routine that computes the $p$-norm~\eqref{e:p}
for not too large values of $p$, while still avoiding overflow of
intermediate results.  For clarity, $\mathtt{xNRMF}$ is described in
detail first, and then $\mathtt{xNRMP}$ is derived from it.

In this work several norm computation algorithms are presented, and
their accuracy and performance are discussed.  Table~1 introduces a
notation for the algorithms to be described in the following, that are
implemented in the two standard floating-point datatypes, with the
associated machine precisions $\varepsilon_{\mathtt{S}}=2^{-24}$ and
$\varepsilon_{\mathtt{D}}=2^{-53}$, due to the assumed rounding to
nearest.  There, $L_{\mathtt{x}}$ stands for the $\mathtt{xNRM2}$
routine.

\begin{table}
  \begin{minipage}{\textwidth}\tablecaptionfont
    Table~1: A categorization of the considered norm computation
    algorithms.  The algorithm $M_{\mathtt{x}}$,
    $M\in\{A,B,C,H,L,X,Y,Z\}$ and
    $\mathtt{x}\in\{\mathtt{S},\mathtt{D}\}$, requires either scalar
    arithmetic or vector registers with $\mathfrak{p}>1$ lanes of the
    corresponding scalar datatype (in C, \texttt{float} for
    $\mathtt{x}=\mathtt{S}$ or \texttt{double} for
    $\mathtt{x}=\mathtt{D}$).
  \end{minipage}\\[1ex]
  \centering
  \begin{tabular}{@{}rcc@{}}\toprule
    \null & scalar & vectorized\\\midrule
    recursive & $A,B,H$ & $X,Y,Z$\\
    iterative & $C,L$ & ---\\
    \bottomrule
  \end{tabular}\hfill
  \begin{tabular}{@{}cr@{}}\toprule
    $M_{\mathtt{x}}$ & $\mathfrak{p}$\\\midrule
    $A_{\mathtt{S}},B_{\mathtt{S}},C_{\mathtt{S}},H_{\mathtt{S}},L_{\mathtt{S}}$ & 1\\
    $A_{\mathtt{D}},B_{\mathtt{D}},C_{\mathtt{D}},H_{\mathtt{D}},L_{\mathtt{D}}$ & 1\\
    \bottomrule
  \end{tabular}\hfill
  \begin{tabular}{@{}cr@{}}\toprule
    $M_{\mathtt{x}}$ & $\mathfrak{p}$\\\midrule
    $X_{\mathtt{S}}$ & 4\\
    $X_{\mathtt{D}}$ & 2\\
    \bottomrule
  \end{tabular}\hfill
  \begin{tabular}{@{}cr@{}}\toprule
    $M_{\mathtt{x}}$ & $\mathfrak{p}$\\\midrule
    $Y_{\mathtt{S}}$ & 8\\
    $Y_{\mathtt{D}}$ & 4\\
    \bottomrule
  \end{tabular}\hfill
  \begin{tabular}{@{}cr@{}}\toprule
    $M_{\mathtt{x}}$ & $\mathfrak{p}$\\\midrule
    $Z_{\mathtt{S}}$ & 16\\
    $Z_{\mathtt{D}}$ & 8\\
    \bottomrule
  \end{tabular}
\end{table}

Based on~\cite{Blue-78,Anderson-17}, $L_{\mathtt{x}}$ maintains the
three accumulators, $\mathtt{sml}$, $\mathtt{med}$, and
$\mathtt{big}$, each of which holds the current, scaled partial sum of
squares of the input elements of a ``small'', ``medium'', or ``big''
magnitude, respectively.  For each $i$, $1\le i\le n$, a small input
element $x_i$ is upscaled, or a big one downscaled, by a suitable
power of two, to prevent under/over-flow, getting $x_i'$, while
$x_i'=x_i^{}$ for a medium $x_i^{}$.  The appropriate accumulator
$\mathtt{acc}$ is then updated, under certain conditions, as
\begin{equation}
  \mathtt{acc}:=\mathtt{acc}+x_i'\cdot x_i',\quad
  \mathtt{acc}\in\{\mathtt{sml},\mathtt{med},\mathtt{big}\},
  \label{e:Lacc}
\end{equation}
what is compiled to a machine equivalent of the C code
$\mathtt{acc}=\mathop{\mathrm{fma[f]}}(x_i',x_i',\mathtt{acc})$, where
$\mathrm{fma}$ denotes the fused multiply-add instruction, with a
single rounding of the result, in double ($\mathrm{fmaf}$ in single)
precision, i.e.,
$\mathop{\mathrm{fma[f]}}(x,y,z)=(x\cdot y+z)(1+\epsilon_{\mathrm{f}})$,
where $|\epsilon_{\mathrm{f}}|\le\varepsilon_{\mathtt{x}}$.  After all
input elements have been processed, $\mathtt{sml}$ and $\mathtt{med}$,
or $\mathtt{med}$ and $\mathtt{big}$, are combined into the final
approximation of $\|\mathbf{x}\|_F$.  If all input elements are of the
medium magnitude, $L_\mathtt{x}$ effectively computes the sum of
squares from~\eqref{e:12I}, \emph{iteratively} from the first to the
last element, using~\eqref{e:Lacc}, and returns its square root.

However, as observed in~\cite[Supplement Sect.~3.1]{Novakovic-26},
$\|\mathbf{x}\|_F$ can be computed without explicitly squaring any
input element.  With the function $\mathrm{hypot[f]}$, defined as
\begin{equation}
  \mathop{\mathrm{hypot[f]}}(x,y)=\sqrt{x^2+y^2}(1+\epsilon_{\mathrm{h}}),
  \label{e:hypot}
\end{equation}
and standardized in the C and Fortran programming languages, it holds
\begin{equation}
  \|[x_1]\|_F=|x_1|,\quad
  \underline{\|[x_1\cdots x_i]\|_F}=\mathop{\mathrm{hypot[f]}}(\underline{\|[x_1\cdots x_{i-1}]\|_F},x_i),\quad
  2\le i\le n,
  \label{e:hi}
\end{equation}
where $\underline{x}$ denotes a floating-point approximation of the
value of the expression $x$.

There are many implementations of $\mathrm{hypot[f]}$ in use, that
differ in accuracy and performance.  A hypotenuse function well suited
for this work's purpose should avoid undue underflow and overflow, be
monotonically non-decreasing with respect to $|x|$ and $|y|$, and be
reasonably accurate, i.e.,
$|\epsilon_{\mathrm{h}}|\le c\varepsilon_{\mathtt{x}}$ for a small
enough $c\ge 1$.  The CORE-MATH project~\cite{Sibidanov-et-al-22} has
developed the \emph{correctly rounded} hypotenuse functions in single,
double, extended, and quadruple precisions\footnote{See
\url{https://core-math.gitlabpages.inria.fr} for further information
and the source code.}.  Such functions, where
$|\epsilon_{\mathtt{h}}|\le\varepsilon_{\mathtt{x}}$, are standardized
as optional in the C language, and are named with the ``cr\_'' prefix,
e.g., $\mathrm{cr\_hypot}$.  Another attempt at developing an accurate
hypotenuse routine is~\cite{Borges-20}.  Some C compilers can be asked
to provide an implementation by the $\mathrm{\_\_builtin\_hypot[f]}$
intrinsic, what might be the C math library's function, possibly
faster than a correctly rounded one.  When not stated otherwise,
$\mathrm{hypot[f]}$ stands for any of those, and for the other scalar
hypotenuse functions to be introduced here.

If instead of two scalars, $x$ and $y$, two vectors $\mathsf{x}$ and
$\mathsf{y}$, each with $\mathfrak{p}>1$ lanes, are given, then
$\mathfrak{p}$ scalar hypotenuses can be computed in parallel, in the
SIMD (Single Instruction, Multiple Data) fashion, such that a new
vector $\mathsf{h}$ is formed as
\begin{equation}
  \mathsf{h}=\mathop{\mathrm{v}\mathfrak{p}\mathrm{\_hypot[f]}}(\mathsf{x},\mathsf{y}),\qquad
  \mathsf{h}_{\ell}=\mathop{\mathrm{v1\_hypot[f]}}(\mathsf{x}_{\ell},\mathsf{y}_{\ell}),\quad
  1\le\ell\le\mathfrak{p},
  \label{e:vh}
\end{equation}
where $\ell$ indexes the vector lanes, and $\mathrm{v1\_hypot[f]}$
denotes an operation that approximates the hypotenuse of the scalars
$\mathsf{x}_{\ell}$ and $\mathsf{y}_{\ell}$ from each lane.  This
operation has to be carefully implemented to avoid branching.  A
vectorized hypotenuse function
$\mathrm{v}\mathfrak{p}\mathrm{\_hypot[f]}$ can be thought of as
applying $\mathrm{v1\_hypot[f]}$ independently and simultaneously to
$\mathfrak{p}$ pairs of scalar inputs.  The Intel's C/C++ compiler
offers such intrinsics; e.g., in double precision with the AVX-512F
vector instruction set (and thus $\mathfrak{p}=8$),
\begin{displaymath}
  \mathtt{\_\_m512d}\ \mathsf{x},\mathsf{y};\quad
  \mathop{\mathrm{v8\_hypot}}(\mathsf{x},\mathsf{y})=\mathop{\mathtt{\_mm512\_hypot\_pd}}(\mathsf{x},\mathsf{y}),
\end{displaymath}
but its exact $\mathrm{v1\_hypot}$ operation is not public, and
therefore cannot be easily ported to other platforms by independent
parties, unlike the vectorized hypotenuse from the SLEEF
library~\cite{Shibata-Petrogalli-20} or the similar one
from~\cite{Novakovic-23}, which is adapted to the
$\text{SSE2}+\text{FMA}$ and $\text{AVX2}+\text{FMA}$ instruction
sets, alongside the AVX-512F, in the following.

Note that~\eqref{e:hi} is a special case of a more general relation.
Let\footnote{Here, and until the $p$-norms are discussed in
Section~\ref{s:4}, the symbol $p$ is used unrelatedly to them.}
$\{i_1,\ldots,i_p\}$ and $\{j_1,\ldots,j_q\}$ be such that $p+q=n$,
$1\le i_k\ne j_l\le n$, $1\le k\le p$, $1\le l\le q$.  Then,
\begin{equation}
  \underline{\|[x_1\cdots x_n]\|_F}=\mathop{\mathrm{hypot[f]}}(\underline{\|[x_{i_1}\cdots x_{i_p}]\|_F},\underline{\|[x_{j_1}\cdots x_{j_q}]\|_F}),
  \label{e:hr}
\end{equation}
what follows from~\eqref{e:hypot}.  In turn,
$\underline{\|[x_{i_1}\cdots x_{i_p}]\|_F}$ and
$\underline{\|[x_{j_1}\cdots x_{j_q}]\|_F}$ can be computed the same,
\emph{recursive} way, until $p$ and $q$ become one or two, when either
the absolute value of the only element is returned, or~\eqref{e:hypot}
is employed, respectively.  In the other direction, \eqref{e:hr} shows
that two partial norms, i.e., the norms of two disjoint subarrays, can
be combined into the norm of the whole array by taking the
$\mathrm{hypot[f]}$ of them.

In Section~\ref{s:2} the roundoff error accumulation in~\eqref{e:Lacc}
and~\eqref{e:hi} is analyzed and it is shown that both approaches
suffer from the similar numerical issues as $n$ grows.  This motivates
the introduction of the recursive scalar algorithms based
on~\eqref{e:hr}, that have substantially tighter relative error bounds
than those of the iterative algorithms, but are inevitably slower than
them.  To improve the performance, the recursive algorithm
$H$ is vectorized in Section~\ref{s:3} as $Z$, which, paired with $A$
for the final reduction, is the proposal for $\mathtt{xNRMF}$.
Another option for thread-based parallelization of the recursive
algorithms, apart from the OpenCilk~\cite{Schardl-Lee-23} one, briefly
described in the previous section, is also presented.
Section~\ref{s:4} shows how to compute the vector $p$-norm by
generalizing $\mathtt{xNRMF}$ to $\mathtt{xNRMP}$.  The numerical
testing in Section~\ref{s:5} confirms the benefits of using the widest
vector registers and relates the performance of $\mathtt{xNRMF}$ to
$L$, the Intel's $\mathtt{xNRM2}$ routine from the MKL
library\footnote{\url{https://www.intel.com/content/www/us/en/developer/tools/oneapi/onemkl.html}},
and the $\mathtt{reproBLAS\_xnrm2}$ from the
ReproBLAS~\cite{Ahrens-et-al-20}
library\footnote{\url{https://github.com/willow-ahrens/ReproBLAS}},
the latter two being the state-of-the-art approaches to the norm
computation.  Section~\ref{s:6} concludes the paper.

Alongside Table~1, the norm computation algorithms that are, to the
best of the author's knowledge, newly proposed here, can also be
summarized as in Table~2.

\begin{table}
  \begin{minipage}{\textwidth}\centering\tablecaptionfont
    Table~2: The recursive algorithms, classified according to the
    $\mathrm{hypot[f]}$ function used in them.
  \end{minipage}\\[1ex]
  \begin{tabular}{@{}rr@{}}\toprule
    $M_{\mathtt{S}}$ & $\mathrm{hypotf}$\\\midrule
    $A_{\mathtt{S}}$ & $\mathrm{cr\_hypotf}$\\
    $B_{\mathtt{S}}$ & $\mathrm{\_\_builtin\_hypotf}$\\
    $H_{\mathtt{S}}$ & $\mathrm{v1\_hypotf}$\\
    \bottomrule
  \end{tabular}\hfill
  \begin{tabular}{@{}rr@{}}\toprule
    $M_{\mathtt{D}}$ & $\mathrm{hypot}$\\\midrule
    $A_{\mathtt{D}}$ & $\mathrm{cr\_hypot}$\\
    $B_{\mathtt{D}}$ & $\mathrm{\_\_builtin\_hypot}$\\
    $H_{\mathtt{D}}$ & $\mathrm{v1\_hypot}$\\
    \bottomrule
  \end{tabular}\hfill
  \begin{tabular}{@{}rr@{}}\toprule
    $M_{\mathtt{S}}$ & $\mathrm{hypotf}$\\\midrule
    $X_{\mathtt{S}}$ & $\mathrm{v4\_hypotf}$\\
    $Y_{\mathtt{S}}$ & $\mathrm{v8\_hypotf}$\\
    $Z_{\mathtt{S}}$ & $\mathrm{v16\_hypotf}$\\
    \bottomrule
  \end{tabular}\hfill
  \begin{tabular}{@{}rr@{}}\toprule
    $M_{\mathtt{D}}$ & $\mathrm{hypot}$\\\midrule
    $X_{\mathtt{D}}$ & $\mathrm{v2\_hypot}$\\
    $Y_{\mathtt{D}}$ & $\mathrm{v4\_hypot}$\\
    $Z_{\mathtt{D}}$ & $\mathrm{v8\_hypot}$\\
    \bottomrule
  \end{tabular}
\end{table}
%
%
\section{Motivation for the recursive algorithms by a roundoff analysis}\label{s:2}
%
%
Under a simplifying assumption that only the $\mathtt{med}$
accumulator is used in $L_{\mathtt{x}}$, Theorem~\ref{thm:1} gives
bounds for the relative error in the obtained approximation
$\underline{\|\mathbf{x}\|_F}$.

\begin{theorem}\label{thm:1}
  Let $\mathbf{x}=[x_1\cdots x_n]$ be an array of finite values in the
  precision $\mathtt{x}$, and $\|\mathbf{x}\|_F$ its Frobenius norm.
  Denote the floating-point square root function by
  $\mathrm{sqrt[f]}$.  If an approximation of
  $\|\mathbf{x}\|_F=\sqrt{g_n}$ is computed as
  $\underline{\|\mathbf{x}\|_F}=\mathop\mathrm{sqrt[f]}(\underline{g_n})$,
  where
  \begin{displaymath}
    \underline{g_0^{}}=g_0^{}=0,\qquad
    g_i^{}=g_{i-1}^{}+x_i^2,\quad
    \underline{g_i^{}}=\mathop{\mathrm{fma[f]}}(x_i^{},x_i^{},\underline{g_{i-1}^{}}),\quad
    1\le i\le n,
  \end{displaymath}
  as in~\eqref{e:Lacc}, then, barring any overflow and inexact underflow, when
  $x_i\ne 0$ it holds
  \begin{equation}
    \underline{g_i^{}}=g_i^{}(1+\eta_i^{}),\quad
    1+\eta_i^{}=\left(1+\eta_{i-1}^{}\frac{g_{i-1}^{}}{g_i^{}}\right)(1+\eta_i'),\quad
    1\le i\le n,
    \label{e:11}
  \end{equation}
  where $|\eta_i'|\le\varepsilon_{\mathtt{x}}^{}$.  With
  $\epsilon_{\sqrt{}}^{}$ such that
  $|\epsilon_{\sqrt{}}^{}|\le\varepsilon_{\mathtt{x}}^{}$, it follows
  \begin{equation}
    \underline{\|\mathbf{x}\|_F^{}}=\mathop{\mathrm{sqrt[f]}}(\underline{g_n^{}})=\|\mathbf{x}\|_F^{}\sqrt{1+\eta_n^{}}(1+\epsilon_{\sqrt{}}^{}),
    \label{e:12}
  \end{equation}
  while the relative error factors from~\eqref{e:11} and~\eqref{e:12}
  can be bounded as
  \begin{equation}
    \begin{gathered}
      1+\eta_i^-=(1+\eta_{i-1}^-)(1-\varepsilon_{\mathtt{x}}^{})\le 1+\eta_i^{}\le(1+\eta_{i-1}^+)(1+\varepsilon_{\mathtt{x}}^{})=1+\eta_i^+,\\
      \sqrt{1+\eta_n^-}(1-\varepsilon_{\mathtt{x}}^{})\le\sqrt{1+\eta_n^{}}(1+\epsilon_{\sqrt{}}^{})\le\sqrt{1+\eta_n^+}(1+\varepsilon_{\mathtt{x}}^{}).
    \end{gathered}
    \label{e:13}
  \end{equation}
\end{theorem}
\begin{proof}
  For $i=1$ \eqref{e:11} holds trivially.  Assume that it holds for
  all $1\le j<i$.  Then,
  \begin{equation}
    \underline{g_{i-1}^{}}+x_i^2=g_{i-1}^{}(1+\eta_{i-1}^{})+x_i^2=(g_{i-1}^{}+x_i^2)(1+d),
    \label{e:14}
  \end{equation}
  where $d$ is found from the second equation in~\eqref{e:14} as
  \begin{displaymath}
    d=\eta_{i-1}^{}\frac{g_{i-1}^{}}{g_{i-1}^{}+x_i^2}=\eta_{i-1}^{}\frac{g_{i-1}^{}}{g_i^{}},
  \end{displaymath}
  so
  $\underline{g_i^{}}=(\underline{g_{i-1}^{}}+x_i^2)(1+\eta_i')=(g_{i-1}^{}+x_i^2)(1+d)(1+\eta_i')=g_i^{}(1+\eta_i^{})$,
  what proves~\eqref{e:11}, and consequently~\eqref{e:12}, with the
  factor $1+\eta_i^{}=(1+d)(1+\eta_i')$.  Its bounds in~\eqref{e:13},
  computable iteratively from $i=1$ to $n$, follow from
  $0\le g_{i-1}\le g_i$ in~\eqref{e:11}.
\end{proof}

If the same classification of the input elements by their magnitude is
used as in $L_{\mathtt{x}}$, and the associated partial norms,
\textsc{sml}, \textsc{med}, and \textsc{big}, are each accumulated as
in~\eqref{e:hi} with $\mathrm{hypot[f]}=\mathrm{cr\_hypot[f]}$, such
an iterative algorithm is called $C_{\mathtt{x}}$.  The separate
accumulators are employed not for the under/over-flow avoidance as in
$L_{\mathtt{x}}$, since unwarranted overflow cannot happen with
$\mathrm{cr\_hypot[f]}$ save for a possible sequence of unfavorable
upward roundings, but for accuracy, to collect the partial norms of
the smaller elements separately, each of which in isolation might
not otherwise affect the partial norm accumulated thus far, should it
become too large.  Finally, the accumulators' values are combined as
$\underline{\|\mathbf{x}\|_F}=\mathop{\mathrm{cr\_hypot[f]}}(\mathop{\mathrm{cr\_hypot[f]}}(\text{\textsc{sml}},\text{\textsc{med}}),\text{\textsc{big}})$,
due to~\eqref{e:hr}.  If only one accumulator is used (e.g.,
\textsc{med}), Theorem~\ref{thm:2} gives bounds for the relative error
in each partial norm and in $\underline{\|\mathbf{x}\|_F}$, computed
by $C_{\mathtt{x}}$ as in~\eqref{e:hi}.

\begin{theorem}\label{thm:2}
  Let $\mathbf{x}=[x_1\cdots x_n]$ be an array of finite values in the
  precision $\mathtt{x}$, and $\|\mathbf{x}\|_F$ its Frobenius norm.
  If its approximation is computed as
  $\underline{\|\mathbf{x}\|_F}=\underline{f_n}$, where
  \begin{displaymath}
    \underline{f_1^{}}=f_1^{}=|x_1^{}|,\qquad
    f_i^{}=\sqrt{f_{i-1}^2+x_i^2},\quad
    \underline{f_i^{}}=\mathop{\mathrm{hypot[f]}}(\underline{f_{i-1}^{}},x_i^{}),\quad
    2\le i\le n,
  \end{displaymath}
  as in~\eqref{e:hi}, then, barring any overflow and inexact
  underflow, when $x_i\ne 0$ it holds
  \begin{equation}
    \underline{f_i^{}}=f_i^{}(1+\epsilon_i^{}),\quad
    1+\epsilon_i^{}=\sqrt{1+\epsilon_{i-1}^{}(2+\epsilon_{i-1}^{})\frac{f_{i-1}^2}{f_i^2}}(1+\epsilon_i'),\quad
    1\le i\le n,
    \label{e:2}
  \end{equation}
  with
  $|\epsilon_i'|\le\varepsilon_{\mathtt{x}}'=c\varepsilon_{\mathtt{x}}^{}$,
  for some $c\ge 1$, defining additionally $\underline{f_0^{}}=f_0^{}=0$
  and $\epsilon_1'=0$.

  Assume that $\mathrm{hypot[f]}$ is $\mathrm{cr\_hypot[f]}$.  Then,
  $\varepsilon_{\mathtt{x}}'=\varepsilon_{\mathtt{x}}^{}$.  If
  $x_i=0$, then $\underline{f_i}=\underline{f_{i-1}}$.  If a lower
  bound of $\epsilon_{i-1}^{}$ is denoted by $\epsilon_{i-1}^-$ and an
  upper bound by $\epsilon_{i-1}^+$, with
  $\epsilon_1^-=\epsilon_1^+=0$, then, while
  $0\ge\epsilon_{i-1}^-\ge-1$, the relative error factor
  $1+\epsilon_i^{}$ from~\eqref{e:2} can be bounded as
  $1+\epsilon_i^-\le 1+\epsilon_i^{}\le 1+\epsilon_i^+$, where
  \begin{equation}
    1+\epsilon_i^-=\sqrt{1+\epsilon_{i-1}^-(2+\epsilon_{i-1}^-)}(1-\varepsilon_{\mathtt{x}}),\ \
    1+\epsilon_i^+=\sqrt{1+\epsilon_{i-1}^+(2+\epsilon_{i-1}^+)}(1+\varepsilon_{\mathtt{x}}).
    \label{e:3}
  \end{equation}
\end{theorem}
\begin{proof}
  For $i=1$, \eqref{e:2} holds trivially with $\epsilon_1'=0$.
  Assuming that~\eqref{e:2} holds for all $j$ such that $1\le j<i$,
  where $2\le i\le n$, and that $x_i\ne 0$, from~\eqref{e:hypot} it
  follows
  \begin{equation}
    \underline{f_i^{}}=\sqrt{\underline{f_{i-1}^2}+x_i^2}(1+\epsilon_i')=\sqrt{f_{i-1}^2(1+\epsilon_{i-1}^{})^2+x_i^2}(1+\epsilon_i').
    \label{e:4}
  \end{equation}
  If the term under the square root on the right hand side
  of~\eqref{e:4} is written as
  \begin{equation}
    f_{i-1}^2(1+\epsilon_{i-1}^{})^2+x_i^2=(f_{i-1}^2+x_i^2)(1+a),
    \label{e:5}
  \end{equation}
  then an easy algebraic manipulation gives
  \begin{equation}
    a=\epsilon_{i-1}^{}(2+\epsilon_{i-1}^{})\frac{f_{i-1}^2}{f_{i-1}^2+x_i^2}=\epsilon_{i-1}^{}(2+\epsilon_{i-1}^{})\frac{f_{i-1}^2}{f_i^2}.
    \label{e:6}
  \end{equation}
  Substituting~\eqref{e:5} into~\eqref{e:4} yields
  \begin{displaymath}
    \underline{f_i^{}}=\sqrt{f_{i-1}^2+x_i^2}\sqrt{1+a}(1+\epsilon_i')=f_i^{}\sqrt{1+a}(1+\epsilon_i')=f_i^{}(1+\epsilon_i^{}),
  \end{displaymath}
  where $(1+\epsilon_i^{})=\sqrt{1+a}(1+\epsilon_i')$, as claimed
  in~\eqref{e:2}.  The bounds~\eqref{e:3} on $1+\epsilon_i$ when
  $\mathrm{hypot[f]}$ is $\mathrm{cr\_hypot[f]}$ follow from the fact
  that the function $x\mapsto x(2+x)$ is  monotonically increasing for
  $x\ge-1$ (here, $x=\epsilon_{i-1}$), and from $0\le f_{i-1}\le f_i$.
\end{proof}

By evaluating~\eqref{e:13} and~\eqref{e:3} from $i=1$ to $n$, using
the MPFR library~\cite{Fousse-et-al-07} with $2048$ bits of precision,
such that, for each $i$, $\eta_i^-$ and $\eta_i^+$, or $\epsilon_i^-$
and $\epsilon_i^+$, respectively, are computed, it can be established
that, for $n$ large enough, the relative error bounds on
$C_{\mathtt{x}}$ are approximately twice larger in magnitude than the
ones on $L_{\mathtt{x}}$, where both algorithms are restricted to a
single accumulator.  Therefore, $C_{\mathtt{x}}$ is not considered for
$\mathtt{xNRMF}$.  However, \eqref{e:hr} is valid not only in the case
of splitting the input array of length $n$ into two subarrays of
lengths $p=n-1$ and $q=1$, as in~\eqref{e:hi}, but also when
$p\approx q$.  If $n=2^k$ for some $k\ge 2$, e.g., then taking $p=q$
in~\eqref{e:hr} reduces the initial norm computation problem to two
problems of half the input length each, and recursively so $k-1$
times.  If $n$ is odd, consider $p=q+1$ to keep $p\ge q$.

Let $R_{\mathtt{x}}$ denote a scalar recursive algorithm.  At every
recursion level except the last, $R_{\mathtt{x}}$ splits its input
array into two contiguous subarrays, the left one being by at most one
element longer, and not shorter, than the right one, calls itself on
both subarrays in turn, and combines their norms.  The splitting stops
when the length of the input array is at most two, when its norm is
calculated directly, as illustrated in~\eqref{e:r} for the initial
length $n=7$, i.e., $p=4$ and $q=3$.  The superscripts before the
operations show their completion order, with the bold ones indicating
the leaf operations that read the input elements from memory \emph{in
the array order}, thus exhibiting the same cache-friendly access
pattern as the iterative algorithms.

\begin{equation}
  \begin{gathered}
    \null^8\mathop{R_{\mathtt{x}}}([x_1\,x_2\,x_3\,x_4\,x_5\,x_6\,x_7]),\\[1ex]
    \null^7\mathop{\mathrm{hypot[f]}}(\mathop{R_{\mathtt{x}}}([x_1\,x_2\,x_3\,x_4]),\mathop{R_{\mathtt{x}}}([x_5\,x_6\,x_7])),\\[1ex]
    \null^3\mathop{\mathrm{hypot[f]}}(\mathop{R_{\mathtt{x}}}([x_1\,x_2]),\mathop{R_{\mathtt{x}}}([x_3\,x_4])),\qquad\null^6\mathop{\mathrm{hypot[f]}}(\mathop{R_{\mathtt{x}}}([x_5\,x_6]),\mathop{R_{\mathtt{x}}}([x_7])),\\[1ex]
    \null^{\mathbf{1}}\mathop{\mathrm{hypot[f]}}(x_1,x_2),\qquad\null^{\mathbf{2}}\mathop{\mathrm{hypot[f]}}(x_3,x_4),\qquad\qquad\null^{\mathbf{4}}\mathop{\mathrm{hypot[f]}}(x_5,x_6),\qquad\null^{\mathbf{5}}|x_7|.
  \end{gathered}
  \label{e:r}
\end{equation}

The relative error bounds for the recursive norm computation, as
in~\eqref{e:hr}, are given in Theorem~\ref{thm:3}.  The choice of
$\mathrm{hypot[f]}$ does not have to be the same with each invocation
(e.g., in~\eqref{e:r} the operation~7 might use a different
$\mathrm{hypot[f]}$ than the rest).

\begin{theorem}\label{thm:3}
  Assume that $\underline{f_{[p]}}=f_{[p]}(1+\epsilon_{[p]})$ and
  $\underline{f_{[q]}}=f_{[q]}(1+\epsilon_{[q]})$ approximate the
  Frobenius norms of some arrays of length $p\ge 1$ and $q\ge 1$,
  respectively, and let
  \begin{displaymath}
    f_{[n]}^2=\sqrt{f_{[p]}^2+f_{[q]}^2},\quad
    \underline{f_{[n]}^{}}=\mathop{\mathrm{hypot[f]}}(\underline{f_{[p]}^{}},\underline{f_{[q]}^{}}),
  \end{displaymath}
  where $\underline{f_{[n]}}$ approximates the Frobenius norm of the
  concatenation of length $n=p+q$ of those arrays, as in~\eqref{e:hr}.
  Then, barring any overflow and inexact underflow, with
  \begin{displaymath}
    1+\epsilon_{[l]}=\min\{1+\epsilon_{[p]},1+\epsilon_{[q]}\},\quad
    1+\epsilon_{[k]}=\max\{1+\epsilon_{[p]},1+\epsilon_{[q]}\},\quad
    1+\epsilon_/=\frac{1+\epsilon_{[l]}}{1+\epsilon_{[k]}},
  \end{displaymath}
  i.e., $l=p$ and $k=q$ or $l=q$ and $k=p$, for $\underline{f_{[n]}}$
  when $f_{[n]}>0$ it holds
  \begin{equation}
    \underline{f_{[n]}^{}}=f_{[n]^{}}(1+\epsilon_{[n]}^{}),\quad
    1+\epsilon_{[n]}^{}=\sqrt{1+\epsilon_/^{}(2+\epsilon_/^{})\frac{f_{[l]}^2}{f_{[n]}^2}}(1+\epsilon_{[k]}^{})(1+\epsilon_{[n]}'),
    \label{e:7}
  \end{equation}
  where
  $|\epsilon_{[n]}'|\le\varepsilon_{\mathtt{x}}'=c\varepsilon_{\mathtt{x}}^{}$,
  for some $c\ge 1$, with $c=1$ if $\mathrm{hypot[f]}$ is
  $\mathrm{cr\_hypot[f]}$.

  If
  $0\le 1+\epsilon_{[i]}^-\le 1+\epsilon_{[i]}^{}\le 1+\epsilon_{[i]}^+$
  for all $i$, $1\le i<n$, then, with
  \begin{equation}
    \begin{aligned}
      1+\epsilon_{[n]}^-&=\sqrt{1+\epsilon_/^-(2+\epsilon_/^-)}(1+\epsilon_{[k]}^-)(1-\varepsilon_{\mathtt{x}}'),\quad
      1+\epsilon_/^-=\frac{1+\epsilon_{[l]}^-}{1+\epsilon_{[k]}^+},\\
      1+\epsilon_{[n]}^+&=\sqrt{1+\epsilon_/^+(2+\epsilon_/^+)}(1+\epsilon_{[k]}^+)(1+\varepsilon_{\mathtt{x}}'),\quad
      1+\epsilon_/^+=\frac{1+\epsilon_{[l]}^+}{1+\epsilon_{[k]}^-},
    \end{aligned}
    \label{e:8}
  \end{equation}
  the relative error in~\eqref{e:7} can be bounded as
  $1+\epsilon_{[n]}^-\le 1+\epsilon_{[n]}^{}\le 1+\epsilon_{[n]}^+$.
\end{theorem}
\begin{proof}
  Expanding $\underline{f_{[p]}^2}+\underline{f_{[q]}^2}$ gives
  \begin{equation}
    \underline{f_{[p]}^2}+\underline{f_{[q]}^2}=f_{[p]}^2(1+\epsilon_{[p]}^{})^2+f_{[q]}^2(1+\epsilon_{[q]}^{})^2=(f_{[l]}^2(1+\epsilon_/^{})^2+f_{[k]}^2)(1+\epsilon_{[k]}^{})^2.
    \label{e:9}
  \end{equation}
  Similarly to~\eqref{e:5}, expressing the first factor on the right hand
  side of~\eqref{e:9} as
  \begin{equation}
    f_{[l]}^2(1+\epsilon_/^{})^2+f_{[k]}^2=(f_{[l]}^2+f_{[k]}^2)(1+b)
    \label{e:10}
  \end{equation}
  leads to
  \begin{displaymath}
    b=\epsilon_/(2+\epsilon_/)\frac{f_{[l]}^2}{f_{[l]}^2+f_{[k]}^2}=\epsilon_/^{}(2+\epsilon_/^{})\frac{f_{[l]}^2}{f_{[n]}^2},
  \end{displaymath}
  and therefore, by substituting~\eqref{e:10} into~\eqref{e:9},
  \begin{displaymath}
    \underline{f_{[n]}^{}}=\sqrt{\underline{f_{[p]}^2}+\underline{f_{[q]}^2}}(1+\epsilon_{[n]}')=\sqrt{f_{[p]}^2+f_{[q]}^2}\sqrt{1+b}(1+\epsilon_{[k]}^{})(1+\epsilon_{[n]}'),
  \end{displaymath}
  what is equivalent to~\eqref{e:7}, while~\eqref{e:8} follows from
  $f_{[l]}\le f_{[n]}$ and, as in the proof of Theorem~\ref{thm:2},
  from the fact that the function $x\mapsto x(2+x)$ is monotonically
  increasing for $x\ge-1$.  With $p$ and $q$ (and thus $n$) given,
  \eqref{e:8} can be computed recursively.
\end{proof}

Listing~1 formalizes the $R_{\mathtt{D}}$ class of algorithms
($R_{\mathtt{S}}$ requires the substitutions
$\text{\texttt{double}}\mapsto\text{\texttt{float}}$, 
$\mathrm{fabs}\mapsto\mathrm{fabsf}$, and
$\mathrm{hypot}\mapsto\mathrm{hypotf}$).  The algorithm
$A_{\mathtt{x}}$ is obtained in the case of
$\mathrm{hypot[f]}=\mathrm{cr\_hypot[f]}$, the algorithm
$B_{\mathtt{x}}$ with $\mathrm{\_\_builtin\_hypot[f]}$, and the
algorithm $H_{\mathtt{x}}$ with $\mathrm{v1\_hypot[f]}$, formalized in
Listing~2 following~\cite[Eq.~(2.13)]{Novakovic-23} for
$\mathtt{x}=\mathtt{D}$ (see also the SLEEF's\footnote{Build the code
from \url{https://github.com/shibatch/sleef} and look into
\texttt{sleefinline\_avx512f.h}.} routine
$\mathtt{Sleef\_hypotd8\_u35avx512f}$).  Note that
$\mathrm{v1\_hypot[f]}$ requires no branching and each of its
statements corresponds to a single arithmetic instruction.  It can be
shown~\cite[Lemma~2.1]{Novakovic-23} that for its relative error
factor $1+\epsilon_{\mathtt{x}}'$, in the notation of
Theorem~\ref{thm:3}, holds
$1+\epsilon_{\mathtt{x}}^{\prime-}<1+\epsilon_{\mathtt{x}}'<1+\epsilon_{\mathtt{x}}^{\prime+}$,
where
\begin{equation}
  1+\epsilon_{\mathtt{x}}^{\prime-}=(1-\varepsilon_{\mathtt{x}}^{})^{\frac{5}{2}}\sqrt{1-\frac{\varepsilon_{\mathtt{x}}^{}(2-\varepsilon_{\mathtt{x}}^{})}{2}},\quad
  1+\epsilon_{\mathtt{x}}^{\prime+}=(1+\varepsilon_{\mathtt{x}}^{})^{\frac{5}{2}}\sqrt{1+\frac{\varepsilon_{\mathtt{x}}^{}(2+\varepsilon_{\mathtt{x}}^{})}{2}}.
  \label{e:reH}
\end{equation}

\begin{table}
\begin{minipage}{\textwidth}\centering\tablecaptionfont
  Listing~1: The $R_{\mathtt{D}}$ class of algorithms in OpenCilk C\@.
\end{minipage}
\begin{lstlisting}
double $R_{\mathtt{D}}$(const $\hbox{\rm\sc integer}$ *const n, const double *const x) { // assume $\text{*n}>0$
  if (*n == 1) return __builtin_fabs(*x); // $|\text{x[0]}|$
  if (*n == 2) return hypot(x[0], x[1]); // one of the described $\mathrm{hypot}$ functions
  const $\hbox{\rm\sc integer}$ p = ((*n >> 1) + (*n & 1)); // $p=\left\lceil n/2\right\rceil\ge 2$
  const $\hbox{\rm\sc integer}$ q = (*n - p); // $q=n-p\le p$
  double fp, fq; // $f_{[p]}$ and $f_{[q]}$
  $\hbox{\rm\sc cilk\_scope}$ { // $\hbox{\rm\sc cilk\_scope}$ is $\text{cilk\_scope}$ if OpenCilk is used, ignored otherwise
    fp = $\!\hbox{\rm\sc cilk\_spawn}$ $R_{\mathtt{D}}$(&p, x); // call $R_{\mathtt{D}}$ recursively on $\mathbf{x}_p=[x_1\cdots x_p]$
    fq = $\!R_{\mathtt{D}}$(&q, (x + p)); // call $R_{\mathtt{D}}$ recursively on $\mathbf{x}_q=[x_{p+1}\cdots x_n]$
  } // $\hbox{\rm\sc cilk\_spawn}$ is $\text{cilk\_spawn}$ if OpenCilk is used, ignored otherwise
  return hypot(fp, fq); // having computed $f_{[p]}^{}$ and $f_{[q]}^{}$, return $f_{[n]}^{}\approx\sqrt{f_{[p]}^2+f_{[q]}^2}$
} // $\hbox{\rm\sc integer}$ corresponds to the Fortran $\text{INTEGER}$ type (e.g., $\text{int}$)
\end{lstlisting}
\end{table}

Listing~1 also shows how to optionally parallelize the scalar
recursive algorithms using\footnote{As described on
\url{https://www.opencilk.org}, OpenCilk is only offered with a
modified Clang C/C++ compiler.  Most of the testing here was thus
performed without OpenCilk, using \texttt{gcc}.} the task parallelism
of OpenCilk.  A function invocation with \texttt{cilk\_spawn}
indicates that the function may, but does not have to, be executed
concurrently with the rest of the code in the same
\texttt{cilk\_scope}.  A scope cannot be exited until all computations
spawned within it have completed, i.e., all their results are
available.

\begin{table}
\begin{minipage}{\textwidth}\centering\tablecaptionfont
  Listing~2: The $\mathrm{v1\_hypot}$ operation in C\@.
\end{minipage}
\begin{lstlisting}
static inline double v1_hypot(const double x, const double y) {
  const double X = __builtin_fabs(x); // $\text{X}=|\text{x}|$
  const double Y = __builtin_fabs(y); // $\text{Y}=|\text{y}|$
  const double m = __builtin_fmin(X, Y); // $\text{m}=\min\{\text{X},\text{Y}\}$
  const double M = __builtin_fmax(X, Y); // $\text{M}=\max\{\text{X},\text{Y}\}$
  const double q = (m / M); // might be a $\text{NaN}$ if, e.g., $\text{m}=\text{M}=0$, but$\ldots$
  const double Q = __builtin_fmax(q, 0.0); // $\ldots\text{Q}$ should not be a $\text{NaN}$
  const double S = __builtin_fma(Q, Q, 1.0); // $\text{S}=\mathop{\mathrm{fma}}(\text{Q},\text{Q},1.0)$
  const double s = __builtin_sqrt(S); // $\text{s}=\mathop{\mathrm{sqrt}}(\text{S})$
  return (M * s); // $\text{M}\sqrt{1+(\text{m}/\text{M})^2}\approx\sqrt{\text{x}^2+\text{y}^2}$
} // if one argument of $\mathrm{fmin}$ or $\mathrm{fmax}$ is a $\text{NaN}$, the other argument is returned
\end{lstlisting}
\end{table}

\looseness=-1
Evaluating~\eqref{e:13} and~\eqref{e:8}, the latter by recursively
computing $\epsilon_{[i]}^-$ and $\epsilon_{[i]}^+$, shows that the
lower bounds on the algorithms' relative errors,
$\mathop{\mathrm{lb\,relerr}}[M_{\mathtt{x}}]$,
\begin{displaymath}
  \mathop{\mathrm{lb\,relerr}}[L_{\mathtt{x}}^{}]=\left(\sqrt{1+\eta_n^-}(1-\varepsilon_{\mathtt{x}}^{})-1\right)/\varepsilon_{\mathtt{x}}^{},\quad
  \mathop{\mathrm{lb\,relerr}}[R_{\mathtt{x}}^{}]=\epsilon_{[n]}^-/\varepsilon_{\mathtt{x}}^{},
\end{displaymath}
are slightly smaller by magnitude than the upper bounds,
$\mathop{\mathrm{ub\,relerr}}[M_{\mathtt{x}}]$,
\begin{equation}
  \mathop{\mathrm{ub\,relerr}}[L_{\mathtt{x}}^{}]=\left(\sqrt{1+\eta_n^+}(1+\varepsilon_{\mathtt{x}}^{})-1\right)/\varepsilon_{\mathtt{x}}^{},\quad
  \mathop{\mathrm{ub\,relerr}}[R_{\mathtt{x}}^{}]=\epsilon_{[n]}^+/\varepsilon_{\mathtt{x}}^{},
  \label{e:ubr}
\end{equation}
and thus it suffices to present only the latter.  The bounds on the
relative error of the underlying $\mathrm{hypot[f]}$ cause
$\epsilon_{[n]}^+$ to be greater for $H$ than for $A$, due
to~\eqref{e:reH}.  Since the bounds on
$\mathrm{\_\_builtin\_hypot[f]}$ depend on the compiler and its math
library (here, the GNU's \texttt{gcc} and \texttt{glibc} were used,
respectively), $B_{\mathtt{x}}$ is excluded from this analysis, but
the math libraries might eventually adopt the correctly rounded
$\mathrm{hypot[f]}$ implementations if their performance is
acceptable, and thus $A$ and $B$ will be the same.

Table~3 shows $\mathop{\mathrm{ub\,relerr}}[M_{\mathtt{D}}]$
from~\eqref{e:ubr} for $M\in\{L,A,H\}$ and $n=2^k$, where
$1\le k\le 30$.  It is evident that the growth in the relative error
bound is \emph{linear} in $n$ for $L_{\mathtt{D}}$ and
\emph{\textbf{logarithmic}} for $A_{\mathtt{D}}$ and $H_{\mathtt{D}}$.
The introduction of the scalar recursive algorithms is thus justified,
even though a quick analysis of Listing~1 can prove they have to be
slower than $L_{\mathtt{x}}$ due to the recursion overhead and a much
higher complexity of $\mathrm{hypot[f]}$, however implemented,
compared to the hardware's $\mathrm{fma[f]}$.

\begin{table}
  \begin{minipage}{\textwidth}\centering\tablecaptionfont
    Table~3: Upper bounds~\eqref{e:ubr} on the relative errors for
    $L_{\mathtt{D}}$, $A_{\mathtt{D}}$, and $H_{\mathtt{D}}$, with
    respect to $n$.
  \end{minipage}\\[1ex]
  \centering\tablefont
  \begin{tabular}{@{}rccc@{}}\toprule
    $\lg n$ & $\mathop{\mathrm{ub\,relerr}}[L_{\mathtt{D}}]$ & $\mathop{\mathrm{ub\,relerr}}[A_{\mathtt{D}}]$ & $\mathop{\mathrm{ub\,relerr}}[H_{\mathtt{D}}]$\\\midrule
     $1$ & $1.50000000000000004\cdot 10^0$ & $1.00000000000000000\cdot 10^0$ & $3.00000000000000036\cdot 10^0$\\
     $2$ & $2.50000000000000021\cdot 10^0$ & $2.00000000000000011\cdot 10^0$ & $6.00000000000000172\cdot 10^0$\\
     $3$ & $4.50000000000000087\cdot 10^0$ & $3.00000000000000033\cdot 10^0$ & $9.00000000000000408\cdot 10^0$\\
     $4$ & $8.50000000000000354\cdot 10^0$ & $4.00000000000000067\cdot 10^0$ & $1.20000000000000074\cdot 10^1$\\
     $5$ & $1.65000000000000142\cdot 10^1$ & $5.00000000000000111\cdot 10^0$ & $1.50000000000000118\cdot 10^1$\\
     $6$ & $3.25000000000000568\cdot 10^1$ & $6.00000000000000167\cdot 10^0$ & $1.80000000000000172\cdot 10^1$\\
     $7$ & $6.45000000000002274\cdot 10^1$ & $7.00000000000000233\cdot 10^0$ & $2.10000000000000235\cdot 10^1$\\
     $8$ & $1.28500000000000909\cdot 10^2$ & $8.00000000000000311\cdot 10^0$ & $2.40000000000000309\cdot 10^1$\\
     $9$ & $2.56500000000003638\cdot 10^2$ & $9.00000000000000400\cdot 10^0$ & $2.70000000000000392\cdot 10^1$\\
    $10$ & $5.12500000000014552\cdot 10^2$ & $1.00000000000000050\cdot 10^1$ & $3.00000000000000486\cdot 10^1$\\
    $11$ & $1.02450000000005821\cdot 10^3$ & $1.10000000000000061\cdot 10^1$ & $3.30000000000000589\cdot 10^1$\\
    $12$ & $2.04850000000023283\cdot 10^3$ & $1.20000000000000073\cdot 10^1$ & $3.60000000000000703\cdot 10^1$\\
    $13$ & $4.09650000000093132\cdot 10^3$ & $1.30000000000000087\cdot 10^1$ & $3.90000000000000826\cdot 10^1$\\
    $14$ & $8.19250000000372529\cdot 10^3$ & $1.40000000000000101\cdot 10^1$ & $4.20000000000000960\cdot 10^1$\\
    $15$ & $1.63845000000149012\cdot 10^4$ & $1.50000000000000117\cdot 10^1$ & $4.50000000000001103\cdot 10^1$\\
    $16$ & $3.27685000000596046\cdot 10^4$ & $1.60000000000000133\cdot 10^1$ & $4.80000000000001257\cdot 10^1$\\
    $17$ & $6.55365000002384186\cdot 10^4$ & $1.70000000000000151\cdot 10^1$ & $5.10000000000001420\cdot 10^1$\\
    $18$ & $1.31072500000953674\cdot 10^5$ & $1.80000000000000170\cdot 10^1$ & $5.40000000000001594\cdot 10^1$\\
    $19$ & $2.62144500003814697\cdot 10^5$ & $1.90000000000000190\cdot 10^1$ & $5.70000000000001777\cdot 10^1$\\
    $20$ & $5.24288500015258789\cdot 10^5$ & $2.00000000000000211\cdot 10^1$ & $6.00000000000001971\cdot 10^1$\\
    $21$ & $1.04857650006103516\cdot 10^6$ & $2.10000000000000233\cdot 10^1$ & $6.30000000000002174\cdot 10^1$\\
    $22$ & $2.09715250024414063\cdot 10^6$ & $2.20000000000000256\cdot 10^1$ & $6.60000000000002388\cdot 10^1$\\
    $23$ & $4.19430450097656250\cdot 10^6$ & $2.30000000000000281\cdot 10^1$ & $6.90000000000002611\cdot 10^1$\\
    $24$ & $8.38860850390625000\cdot 10^6$ & $2.40000000000000306\cdot 10^1$ & $7.20000000000002844\cdot 10^1$\\
    $25$ & $1.67772165156250000\cdot 10^7$ & $2.50000000000000333\cdot 10^1$ & $7.50000000000003088\cdot 10^1$\\
    $26$ & $3.35544325625000001\cdot 10^7$ & $2.60000000000000361\cdot 10^1$ & $7.80000000000003341\cdot 10^1$\\
    $27$ & $6.71088647500000006\cdot 10^7$ & $2.70000000000000390\cdot 10^1$ & $8.10000000000003605\cdot 10^1$\\
    $28$ & $1.34217729500000005\cdot 10^8$ & $2.80000000000000420\cdot 10^1$ & $8.40000000000003878\cdot 10^1$\\
    $29$ & $2.68435460500000040\cdot 10^8$ & $2.90000000000000451\cdot 10^1$ & $8.70000000000004161\cdot 10^1$\\
    $30$ & $5.36870928500000318\cdot 10^8$ & $3.00000000000000483\cdot 10^1$ & $9.00000000000004455\cdot 10^1$\\
    \bottomrule
  \end{tabular}
\end{table}

\looseness=-1
The single precision error bounds are less informative, as explained
with Figure~\ref{fig:1}.  The tester $T$ is parameterized by $t$,
$\mathtt{x}$, and $\mathcal{D}$, where $t$ is the run number,
$1\le t\le 31$, $\mathtt{x}$ is the chosen precision, and
$\mathcal{D}\in\{\mathop{\mathcal{U}}(0,1),\mathop{\mathcal{N}}(0,1)\}$
is either the uniform or the normal random distribution.  Given $t$
and $\mathcal{D}$, the randomly generated but stored seed
$s_t^{\mathcal{D}}$ is retrieved, and an input array $\mathbf{x}$,
aligned to the cache line size, of $n=2^{29}$ pseudorandom numbers in
the precision $\mathtt{x}$, is generated, what can be done by the
$\mathtt{xLARND}$ routine from LAPACK~\cite{Anderson-et-al-99} with
the arguments $\mathtt{IDIST}=1$ and $\mathtt{IDIST}=3$ for
$\mathop{\mathcal{U}}(0,1)$ and $\mathop{\mathcal{N}}(0,1)$,
respectively, and with the initial $\mathtt{ISEED}=s_t^{\mathcal{D}}$.
Generating the inputs with the relatively small magnitudes of their
elements makes it possible to test the algorithms with large
values\footnote{Up to $n=2^{30}$ has been tried, to meaningfully check
for accuracy and obtain stable timing results.} of $n$ without
necessitating the results' overflow.

The ``exact'' (i.e., representable in $\mathtt{x}$ and as close to
exact as feasible) Frobenius norm $\|\mathbf{x}\|_F'$ is computed
recursively, following $R_{\mathtt{x}}$, but using MPFR with $2048$
bits of precision, and rounding the result to the nearest value
representable in $\mathtt{x}$.  Then, $T$ runs all algorithms under
consideration on $\mathbf{x}$, timing their execution and computing
their relative error with respect to $\|\mathbf{x}\|_F'$.  The
relative error (in multiples of $\varepsilon_{\mathtt{x}}$) of an
algorithm $M_{\mathtt{x}}$ on $\mathbf{x}$ is defined as
\begin{equation}
  \mathop{\mathrm{relerr}[M_{\mathtt{x}}]}(\mathbf{x})=\frac{|\|\mathbf{x}\|_F'-\underline{\|\mathbf{x}\|_F^{}}|}{\|\mathbf{x}\|_F'\cdot\varepsilon_{\mathtt{x}}^{}},\quad
  \underline{\|\mathbf{x}\|_F^{}}=\mathop{M_{\mathtt{x}}^{}}(\mathbf{x}),
  \label{e:relerr}
\end{equation}
where the division by $\varepsilon_{\mathtt{x}}^{}$ makes the relative
errors comparable across both precisions.

\begin{figure}
  \centering
  \includegraphics[keepaspectratio]{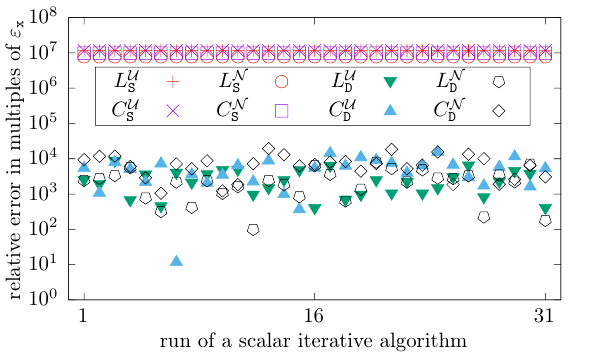}
  \caption{The observed relative errors~\eqref{e:relerr} for $L$ and $C$ in both precisions.}
  \label{fig:1}
\end{figure}

Three important conclusions follow from Figure~\ref{fig:1}.  First, in
single precision, both iterative algorithms can more easily reach a
point where a particular accumulator gets ``saturated'', i.e., so big
that no further update can change its value, regardless of whether it
accumulates the partial norm ($C_{\mathtt{S}}$) or the sum of squares
($L_{\mathtt{S}}$).  Once that happens, the rest of the input elements
of that accumulator's class is effectively ignored.  Second,
$L_{\mathtt{D}}$ and $C_{\mathtt{D}}$ are of comparable but poor
accuracy in the majority of the runs.  Third, the peak relative error
in double precision is about the square root of the upper bound from
Table~3.  But the most important conclusion is not visible in
Figure~\ref{fig:1}.  All scalar and vectorized recursive algorithms,
in both precisions, on the respective inputs have the relative
error~\eqref{e:relerr} less than \emph{three}.  Since the input
elements' magnitudes do not vary widely, at every node of the
recursion tree (see~\eqref{e:r}), the values being returned by its
left and the right branch are not so different that one would not
generally affect the other when combined by $\mathrm{hypot[f]}$.
%
%
\section{Vectorization of the recursive algorithms}\label{s:3}
%
%
It remains to improve the performance of the recursive algorithms,
what can hardly be done without vectorization.  Even though their
structure allows for a thread-based parallelization, such that several
independent recursion subtrees are computed each in their own thread,
the thread management overhead might be too large for any gain in
performance.  For extremely large $n$ a thread-based parallelization
will help, but even then, single-threaded vectorized subrecursions
should run faster than (but with a similar accuracy as) sequential
scalar ones, as demonstrated in the following.

Listing~3 is an implementation of~\eqref{e:vh} for
$\mathtt{x}=\mathtt{D}$ and $\mathfrak{p}=8$, similar
to~\cite[Algorithm~2.1]{Novakovic-23}.  It directly corresponds to
Listing~1 since the $\mathrm{v1\_hypot}$ operation is performed
simultaneously for all $\ell$.  The lines~8 and~9 clear the sign bits
of $\mathsf{x}_{\ell}$ and $\mathsf{y}_{\ell}$, respectively, while
the other operations are the vector variants of the standard C scalar
arithmetic, as provided\footnote{See
\url{https://www.intel.com/content/www/us/en/docs/intrinsics-guide/index.html}.}
by the compiler's intrinsic functions.  It is straightforward to adapt
$\mathrm{v8\_hypot}$ to another $\mathfrak{p}$ and/or $\mathtt{x}$,
and to the other platforms' vector instruction sets.  All arithmetic
is done in vector registers, without branching.

\begin{table}
\begin{minipage}{\textwidth}\centering\tablecaptionfont
  Listing~3: The $\mathrm{v8\_hypot}$ operation in C with AVX-512F.
\end{minipage}
\begin{lstlisting}
#ifndef __AVX512DQ__ // if only AVX512F is available$\ldots$
#define $\hbox{\rm\tt\_mm512\_andnot\_pd}$($\mathsf{b}$, $\mathsf{a}$) _mm512_castsi512_pd(\
  _mm512_andnot_epi64(_mm512_castpd_si512($\mathsf{b}$), _mm512_castpd_si512($\mathsf{a}$)))
#endif // $\ldots$define the $\hbox{\rm\tt\_mm512\_andnot\_pd}$ operation
static inline __m512d v8_hypot($\hbox{\rm\sc reg}$ __m512d $\mathsf{x}$, $\hbox{\rm\sc reg}$ __m512d $\mathsf{y}$) {
  $\hbox{\rm\sc reg}$ __m512d $\mathsf{z}$ = _mm512_set1_pd(-0.0); // $\mathsf{z}_{\ell}=-0.0$
  $\hbox{\rm\sc reg}$ __m512d $\mathsf{o}$ = _mm512_set1_pd(1.0); // $\mathsf{o}_{\ell}=1.0$
  $\hbox{\rm\sc reg}$ __m512d $\mathsf{X}$ = _mm512_andnot_pd($\mathsf{z}$, $\mathsf{x}$); // $\mathsf{X}_{\ell}=\mathsf{x}_{\ell}\mathop{\mathrm{bitand}}(\mathop{\mathrm{bitnot}}{\mathsf{z}_{\ell}})=|\mathsf{x}_{\ell}|$
  $\hbox{\rm\sc reg}$ __m512d $\mathsf{Y}$ = _mm512_andnot_pd($\mathsf{z}$, $\mathsf{y}$); // $\mathsf{Y}_{\ell}=\mathsf{y}_{\ell}\mathop{\mathrm{bitand}}(\mathop{\mathrm{bitnot}}{\mathsf{z}_{\ell}})=|\mathsf{y}_{\ell}|$
  $\hbox{\rm\sc reg}$ __m512d $\mathsf{m}$ = _mm512_min_pd($\mathsf{X}$, $\mathsf{Y}$); // $\mathsf{m}_{\ell}=\min\{\mathsf{X}_{\ell},\mathsf{Y}_{\ell}\}$
  $\hbox{\rm\sc reg}$ __m512d $\mathsf{M}$ = _mm512_max_pd($\mathsf{X}$, $\mathsf{Y}$); // $\mathsf{M}_{\ell}=\max\{\mathsf{X}_{\ell},\mathsf{Y}_{\ell}\}$
  $\hbox{\rm\sc reg}$ __m512d $\mathsf{q}$ = _mm512_div_pd($\mathsf{m}$, $\mathsf{M}$); // $\mathsf{q}_{\ell}=\mathsf{m}_{\ell}/\mathsf{M}_{\ell}$
  $\hbox{\rm\sc reg}$ __m512d $\mathsf{Q}$ = _mm512_max_pd($\mathsf{q}$, $\mathsf{z}$); // $\mathsf{Q}_{\ell}=\mathop{\mathrm{fmax}}(\mathsf{q}_{\ell},\mathsf{z}_{\ell})$
  $\hbox{\rm\sc reg}$ __m512d $\mathsf{S}$ = _mm512_fmadd_pd($\mathsf{Q}$, $\mathsf{Q}$, $\mathsf{o}$); // $\mathsf{S}_{\ell}=\mathop{\mathrm{fma}}(\mathsf{Q}_{\ell},\mathsf{Q}_{\ell},\mathsf{o}_{\ell})$
  $\hbox{\rm\sc reg}$ __m512d $\mathsf{s}$ = _mm512_sqrt_pd($\mathsf{S}$); // $\mathsf{s}_{\ell}=\mathop{\mathrm{sqrt}}(\mathsf{S}_{\ell})$
  $\hbox{\rm\sc reg}$ __m512d $\mathsf{h}$ = _mm512_mul_pd($\mathsf{M}$, $\mathsf{s}$); // $\mathsf{h}_{\ell}=\mathsf{M}_{\ell}\cdot\mathsf{s}_{\ell}$
  return $\mathsf{h}$; // $\mathsf{h}_{\ell}^{}\approx\sqrt{\mathsf{x}_{\ell}^2+\mathsf{y}_{\ell}^2}$, for all lanes $\ell$, $1\le\ell\le\mathfrak{p}=8$
} // $\hbox{\rm\sc reg}$ stands for $\hbox{\rm\tt register const}$
\end{lstlisting}
\end{table}

The input array $\mathbf{x}$ is assumed to reside in a contiguous
memory region with the natural alignment, i.e., each element has an
address that is an integer multiple of the datatype's size in bytes,
$\mathfrak{s}$, and thus can be thought of as consisting of at most
three parts.  The first part is \textsc{head}, possibly empty,
comprising the elements that lie before the first one aligned to the
vector size, i.e., that has an address divisible by
$\mathfrak{p}\cdot\mathfrak{s}$.  A non-empty \textsc{head} means that
$\mathbf{x}$ is not vector-aligned.  The second part is a (possibly
empty) sequence of groups of $\mathfrak{p}$ elements.  The last part,
\textsc{tail}, also possibly empty, is vector-aligned but has fewer
than $\mathfrak{p}$ elements.  Not all three parts are empty, because
$n\ge 1$.  Vector loads from a non-vector-aligned address might be
slower, so the presence of a non-empty \textsc{head} has to be dealt
with somehow.  The simplest but suboptimal solution, that guarantees
the same numerical results with and without \textsc{head}, is to use
the aligned-load instructions when \textsc{head} is empty, and the
unaligned-load ones otherwise.  Algorithms using the former will be
denoted by $\mathfrak{a}$ in the superscript, and those that employ
the latter by $\mathfrak{u}$.  Also, \textsc{tail} has to be loaded in
a special way to avoid accessing the unallocated memory.  Masked
vector loads, e.g., can be used to fill the lowest lanes of a vector
register with the elements of \textsc{tail}, while setting the higher
lanes to zero.  A possible situation with $n=13$ and $\mathfrak{p}=8$,
where \textsc{head} might have, e.g., three, and \textsc{tail} two
elements, is illustrated as
\begin{displaymath}
  \settowidth{\fbw}{\fbox{$x_{10}$}}
  \settowidth{\fbh}{\framebox{\framebox[\fbw][c]{$x_1$}\framebox[\fbw][c]{$x_2$}\framebox[\fbw][c]{$x_3$}}}
  \settowidth{\fba}{\framebox{\framebox[\fbw][c]{$x_4$}\framebox[\fbw][c]{$x_5$}\framebox[\fbw][c]{$x_6$}\framebox[\fbw][c]{$x_7$}\framebox[\fbw][c]{$x_8$}\framebox[\fbw][c]{$x_9$}\framebox[\fbw][c]{$x_{10}$}\framebox[\fbw][c]{$x_{11}$}}}
  \settowidth{\fbt}{\framebox{\framebox[\fbw][c]{$x_{12}$}\framebox[\fbw][c]{$x_{13}$}}}
  \begin{gathered}
    \framebox{\framebox[\fbw][c]{$\mathcolor{Blue}{x_1}$}\framebox[\fbw][c]{$\mathcolor{Blue}{x_2}$}\framebox[\fbw][c]{$\mathcolor{Blue}{x_3}$}}\framebox{\framebox[\fbw][c]{$x_4$}\framebox[\fbw][c]{$x_5$}\framebox[\fbw][c]{$x_6$}\framebox[\fbw][c]{$x_7$}\framebox[\fbw][c]{$x_8$}\framebox[\fbw][c]{$x_9$}\framebox[\fbw][c]{$x_{10}$}\framebox[\fbw][c]{$x_{11}$}}\framebox{\framebox[\fbw][c]{$\mathcolor{Brown}{x_{12}}$}\framebox[\fbw][c]{$\mathcolor{Brown}{x_{13}}$}}\\
    \makebox[\fbh][c]{\text{\textsc{head}}}\makebox[\fba][c]{\text{the aligned full-vector subarray}}\makebox[\fbt][c]{\text{\textsc{tail}}}
  \end{gathered}\,.
\end{displaymath}

The non-unit-stride option of $L_{\mathtt{x}}$, when its argument
$\mathtt{incx}>1$ and the array elements to be accessed are assumed to
be separated by $\mathtt{incx}-1$ elements (so not contiguous), can be
realized here with the vector gather instructions and a possible
associated performance penalty.  Such an option, as well as the one
for $\mathtt{incx}\le-1$, where the elements are accessed in the
opposite order, is left for future work.

Listing~4 specifies the $Z_{\mathtt{D}}^{\mathfrak{a}}$, and comments
on the $Z_{\mathtt{D}}^{\mathfrak{u}}$ algorithm.  For brevity, the
$X$ and $Y$ algorithms are omitted but can easily be deduced, or their
implementation can be looked up in the supplementary material.  The
notation follows Listing~2 and Listing~3.  Structurally,
$Z_{\mathtt{D}}^{\mathfrak{a}}$ checks for the terminating conditions
of the recurrence, deals with \textsc{tail} if required, otherwise
splits $\mathbf{x}$ into two parts, the first one having a certain
number of full, aligned vectors (i.e., no \textsc{tail}), and calls
itself recursively on both parts, similarly to $R_{\mathtt{D}}$.  This
algorithm, however, returns a vector of partial norms, that has to be
reduced further to the final $\underline{\|\mathbf{x}\|_F}$, what is
described separately.

\begin{table}
\begin{minipage}{\textwidth}\centering\tablecaptionfont
  Listing~4: The $Z_{\mathtt{D}}$ algorithm in OpenCilk C\@.
\end{minipage}
\begin{lstlisting}
__m512d $Z_{\mathtt{D}}^{\mathfrak{a}}$(const $\hbox{\rm\sc integer}$ n, const double *const x) { // assume $\text{n}>0$
  register const __m512d $\mathsf{z}$ = _mm512_set1_pd(-0.0); // $\mathsf{z}_{\ell}=-0.0$
  const $\hbox{\rm\sc integer}$ r = (n & 7); // $r=n\bmod\mathfrak{p}$, the number of elements in $\hbox{\rm\sc tail}$
  const $\hbox{\rm\sc integer}$ m = ((n >> 3) + (r != 0)); // $m=\left\lceil n/\mathfrak{p}\right\rceil$
  if (m == 1) { // $1\le n\le\mathfrak{p}$, so there is only one vector, either full ($r=0$) or $\hbox{\rm\sc tail}$
    if (r == 0) return _mm512_andnot_pd($\mathsf{z}$, _mm512_load_pd(x));$^{\dagger}$ // a full vector
    if (r == 1) return _mm512_andnot_pd($\mathsf{z}$, _mm512_mask_load_pd($\mathsf{z}$, 0x01, x));$^{\dagger}$
    if (r == 2) return _mm512_andnot_pd($\mathsf{z}$, _mm512_mask_load_pd($\mathsf{z}$, 0x03, x));$^{\dagger}$
    if (r == 3) return _mm512_andnot_pd($\mathsf{z}$, _mm512_mask_load_pd($\mathsf{z}$, 0x07, x));$^{\dagger}$
    if (r == 4) return _mm512_andnot_pd($\mathsf{z}$, _mm512_mask_load_pd($\mathsf{z}$, 0x0F, x));$^{\dagger}$
    if (r == 5) return _mm512_andnot_pd($\mathsf{z}$, _mm512_mask_load_pd($\mathsf{z}$, 0x1F, x));$^{\dagger}$
    if (r == 6) return _mm512_andnot_pd($\mathsf{z}$, _mm512_mask_load_pd($\mathsf{z}$, 0x3F, x));$^{\dagger}$
    if (r == 7) return _mm512_andnot_pd($\mathsf{z}$, _mm512_mask_load_pd($\mathsf{z}$, 0x7F, x));$^{\dagger}$
  } // if $m=1$ return $[|x_1|\cdots|x_{\mathfrak{p}}|]$, or $|\hbox{\rm\sc tail}|=[|x_1|\cdots|x_r|\,0_{r+1}\cdots 0_{\mathfrak{p}}]$ if $r>0$
  register __m512d fp, fq;
  if (m == 2) { // $\mathfrak{p}+1\le n\le 2\mathfrak{p}$
    fp = _mm512_load_pd(x);$^{\dagger}$ // load the full left vector; if the right one is $\hbox{\rm\sc tail}\ldots$
    fq = (r ? $Z_{\mathtt{D}}^{\mathfrak{a}}$(r, (x + 8)) : _mm512_load_pd(x + 8));$^{\dagger}$ // $\ldots Z_{\mathtt{D}}^{\mathfrak{a}}$ gives $|\hbox{\rm\sc tail}|$
    return v8_hypot(fp, fq);
  } // if $m=2$ return $\mathop{\mathrm{v}\mathfrak{p}\mathrm{\_hypot}}([x_1\cdots x_{\mathfrak{p}}],[x_{\mathfrak{p}+1}\cdots x_{2\mathfrak{p}}])$ or $\mathop{\mathrm{v}\mathfrak{p}\mathrm{\_hypot}}([x_1\cdots x_{\mathfrak{p}}],|\hbox{\rm\sc tail}|)$
  const $\hbox{\rm\sc integer}$ p = (((m >> 1) + (m & 1)) << 3); // $w=\left\lceil m/2\right\rceil\ge 2,\quad p=w\cdot\mathfrak{p}$
  const $\hbox{\rm\sc integer}$ q = (n - p); // $q=n-p\le p$
  $\hbox{\rm\sc cilk\_scope}$ { // optional parallelization with OpenCilk
    fp = $\hbox{\rm\sc cilk\_spawn}$ $Z_{\mathtt{D}}^{\mathfrak{a}}$(p, x); // call $Z_{\mathtt{D}}^{\mathfrak{a}}$ on $\mathbf{x}_p=[x_1\cdots x_p]$ with $w$ full vectors
    fq = $Z_{\mathtt{D}}^{\mathfrak{a}}$(q, (x + p)); // call $Z_{\mathtt{D}}^{\mathfrak{a}}$ on $\mathbf{x}_q=[x_{p+1}\cdots x_n]$ with $m-w$ vectors
  } // $(\mathsf{f}_{[p]})_{\ell}\approx\sqrt{x_{\ell}^2+x_{\ell+\mathfrak{p}}^2+\cdots+x_{\ell+(w-1)\mathfrak{p}}^2},\quad(\mathsf{f}_{[q]})_{\ell}\approx\sqrt{x_{\ell+p}^2+x_{\ell+p+\mathfrak{p}}^2+\cdots}$
  return v8_hypot(fp, fq); // $\mathop{\mathrm{v}\mathfrak{p}\mathrm{\_hypot}}(\mathsf{f}_{[p]},\mathsf{f}_{[q]})$
} // $^{\dagger}Z_{\mathtt{D}}^{\mathfrak{u}}$: if $\mathbf{x}$ is not aligned to $64\,\mathrm{B}$, use $\text{*loadu*}$ instead of the $\text{*load*}$ operations
\end{lstlisting}
\end{table}

The conclusion from~\eqref{e:r} is still valid in the vector case,
i.e., the elements of $\mathbf{x}$ are loaded from memory in the array
order.  However, a partial norm in the lane $\ell$ is computed from
the elements in the same lane, their indices being separated by an
integer multiple of $\mathfrak{p}>1$.  Let, e.g., $\mathfrak{p}=4$
and $n=16$ ($m=4$).  Then, $\mathbf{x}$ might be
\begin{displaymath}
  \mathbf{x}=\begin{bmatrix}
  \fbox{$\mathcolor{Maroon}{x_1}\,\mathcolor{PineGreen}{x_2}\,\mathcolor{Red}{x_3}\,\mathcolor{CadetBlue}{x_4}$}&
  \fbox{$\mathcolor{Maroon}{x_5}\,\mathcolor{PineGreen}{x_6}\,\mathcolor{Red}{x_7}\,\mathcolor{CadetBlue}{x_8}$}&
  \fbox{$\mathcolor{Maroon}{x_9}\,\mathcolor{PineGreen}{x_{10}}\,\mathcolor{Red}{x_{11}}\,\mathcolor{CadetBlue}{x_{12}}$}&
  \fbox{$\mathcolor{Maroon}{x_{13}}\,\mathcolor{PineGreen}{x_{14}}\,\mathcolor{Red}{x_{15}}\,\mathcolor{CadetBlue}{x_{16}}$}
  \end{bmatrix},
\end{displaymath}
assuming \textsc{head} and \textsc{tail} are empty.  The final vector
of partial norms returned is
\begin{displaymath}
  \mathop{Y_{\mathtt{D}}^{}}(16,\mathbf{x})\approx\begin{bmatrix}
  \sqrt{(\mathcolor{Maroon}{x_1}^2+\mathcolor{Maroon}{x_5}^2)+(\mathcolor{Maroon}{x_9}^2+\mathcolor{Maroon}{x_{13}}^2)}\\
  \sqrt{(\mathcolor{PineGreen}{x_2}^2+\mathcolor{PineGreen}{x_6}^2)+(\mathcolor{PineGreen}{x_{10}}^2+\mathcolor{PineGreen}{x_{14}}^2)}\\
  \sqrt{(\mathcolor{Red}{x_3}^2+\mathcolor{Red}{x_7}^2)+(\mathcolor{Red}{x_{11}}^2+\mathcolor{Red}{x_{15}}^2)}\\
  \sqrt{(\mathcolor{CadetBlue}{x_4}^2+\mathcolor{CadetBlue}{x_8}^2)+(\mathcolor{CadetBlue}{x_{12}}^2+\mathcolor{CadetBlue}{x_{16}}^2)}
  \end{bmatrix}^T\!\!\!=
  \left[\left(\sqrt{x_{\ell}^2+x_{\ell+\mathfrak{p}}^2+x_{\ell+2\mathfrak{p}}^2+x_{\ell+3\mathfrak{p}}^2}\right)_{\!\ell}^{}\right],
\end{displaymath}
where $1\le\ell\le\mathfrak{p}$, i.e.,
$\mathcolor{Maroon}{\ell=\text{lane~1}}$ or
$\mathcolor{PineGreen}{\ell=\text{lane~2}}$ or
$\mathcolor{Red}{\ell=\text{lane~3}}$ or
$\mathcolor{CadetBlue}{\ell=\text{lane~4}}$.  The vectorized
algorithms' results from one system, even if the OpenCilk parallelism
is used, are therefore bitwise reproducible on another with the same
$\mathfrak{p}$, but not with a different one.  This is in contrast
with the scalar algorithms, that are always unconditionally
reproducible, except for $B$, which is platform dependent by design.

The recursive algorithms do not require much stack space for their
variables.  Their recursion depth is
$\left\lceil\lg(\max\{n/\mathfrak{p},1\})\right\rceil$, so a stack
overflow is unlikely.

It might be too expensive to enforce any particular order of the
elements within each vector of the partial norms.  However, at least
the final, output vector of $Z_{\mathtt{D}}$ can be sorted
non-decreasingly, without function calls and in the vectorized
fashion, as described in~\cite{Bramas-17}, what might improve
accuracy, by reducing the smaller norms first.  This has been
implemented for $Z_{\mathtt{D}}$, but might be extended to other
routines.

One option for reducing the final value $\mathsf{f}$ of a vectorized
algorithm to $\underline{\|\mathbf{x}\|_F}$ is to split $\mathsf{f}$
into two vectors of the length $\mathfrak{p}/2$, and to compute the
vector $\mathrm{hypot[f]}$ of them, repeating the process until
$\mathfrak{p}=1$.  Schematically, if $\mathfrak{p}=8$, e.g.,
\begin{equation}
  \begin{aligned}
    \mathsf{f}&=[f_1\,f_2\,f_3\,f_4\,f_5\,f_6\,f_7\,f_8]\rightarrow\mathop{\mathrm{v4\_hypot[f]}}([f_1\,f_2\,f_3\,f_4],[f_5\,f_6\,f_7\,f_8])\\
    &\rightarrow[f_1'\,f_2'\,f_3'\,f_4']\rightarrow\mathop{\mathrm{v2\_hypot[f]}}([f_1'\,f_2'],[f_3',f_4'])\\
    &\rightarrow[f_1''\,f_2'']\rightarrow\mathop{\mathrm{v1\_hypot[f]}}(f_1'',f_2'')\rightarrow\underline{\|\mathbf{x}\|_F}.
  \end{aligned}
  \label{e:vred}
\end{equation}
But for a large $n$ the final reduction should not affect the overall
performance much, so it is possibly more accurate to compute the norm
of $\mathsf{f}$, and thus of $\mathbf{x}$, by $A$.  For this,
$\mathsf{f}$ has to be stored from a vector register into a local
array on the stack.

\emph{The recommendation for $\mathtt{xNRMF}$ is to select
$Z_{\mathtt{x}}$, with $A_{\mathtt{x}}$ for the final reduction}.  If
$\mathbf{x}$ is vector-aligned, call $Z_{\mathtt{x}}^{\mathfrak{a}}$,
else call $Z_{\mathtt{x}}^{\mathfrak{u}}$, and reduce the output
vector in either case to $\underline{\|\mathbf{x}\|_F}$ by
$A_{\mathtt{x}}$.  If $\mathrm{cr\_hypot[f]}$ is unavailable, consider
$B_{\mathtt{x}}$ or~\eqref{e:vred} instead of $A_{\mathtt{x}}$.
Similarly, $X$ and $Y$ have to be paired with a final reduction
algorithm $R$.  In the following, $X$, $Y$, and $Z$ are redefined to
stand for those algorithms paired with $A$.

The simplicity of the recursive algorithms allows for a speedup if the
vector length is increased beyond $512$ bits and the routines from
Listings~3 and~4 are re-written accordingly.  A limiting factor might
be the use of one division and one square root for each hypotenuse
operation, what deserves attention in future work.
%
%
\subsection{OpenMP parallelization and multi-dimensional arrays}\label{ss:3.1}
%
%
The recursive algorithms can alternatively be parallelized by
OpenMP~\cite{OMPARB-24}, by splitting the input array to approximately
equally sized contiguous chunks, each of which is given to a different
thread to compute its norm by a vectorized sequential recursive
algorithm.  Then, the final norm is reduced from the threads' partial
ones by noting that $\mathrm{hypot[f]}$ can be used as a user-defined
reduction operator in \texttt{omp declare reduction} directives.
However, since the reduction order is unspecified, the reproducibility
for any fixed number of threads greater than two would in theory be
jeopardized.  In practice, the Intel's OpenMP implementation, e.g.,
allows setting the environment variable
\texttt{KMP\_DETERMINISTIC\_REDUCTION} to \texttt{TRUE}.

The OpenMP parallelization is better suited for computing the
Frobenius norm of a multi-dimensional array.  If all elements of the
array are stored contiguously, then the array can be regarded as
one-dimensional.  If not, e.g., when an $\mathtt{M}\times\mathtt{N}$
matrix $\mathtt{A}$ in the Fortran order is stored with the leading
dimension $\mathtt{LDA}$ larger then the actual number of rows (i.e.,
when $\mathtt{LDA}>\mathtt{M}$ for
$\mathtt{A}(\mathtt{LDA},\mathtt{N})$), then each thread should
compute the norm of a subset of contiguous lower-dimensional
subarrays.  In the matrix example, those subarrays would be the
columns of the matrix, or its rows if it is stored in the C array
order.  The partial norms would then be reduced as described in the
previous paragraph.  This principle applies also for
higher-dimensional arrays: if $\mathtt{A}$ is allocated as
$\mathtt{A}(\mathtt{LDA},\mathtt{N}_2,\mathtt{N}_3,\ldots,\mathtt{N}_k)$
and $\mathtt{LDA}>\mathtt{M}$, then there are
$\mathtt{N}_2\times\mathtt{N}_3\times\cdots\times\mathtt{N}_k$
columns, the norms of which can be computed in parallel and reduced as
described.
%
%
\section{Computation of the vector $p$-norm}\label{s:4}
%
%
A method for computing the vector $p$-norm~\eqref{e:p} can be used to
calculate the unitarily invariant Schatten $p$-norm of a
matrix~\cite{Horn-Johnson-12}, which in turn finds applications in,
e.g., image
reconstruction~\cite{Lefkimmiatis-et-al-13a,Lefkimmiatis-et-al-13b}.
From the singular value decomposition of a matrix $G$ as
$G=U\Sigma V^{\ast}$ and $\mathbf{s}=[\sigma_1\cdots\sigma_n]$, the
Shatten $p$-norm is obtained as $\|G\|_{S_p}=\|\mathbf{s}\|_p$.

In the case of $p=\infty$, the construction of $\mathtt{xNRMP}$ from
$\mathtt{xNRMF}$ is trivial.  The vectorized hypotenuse operation
$\mathrm{v}\mathfrak{p}\mathrm{\_hypot}$ has to be replaced by
$\mathrm{v}\mathfrak{p}\mathrm{\_maxabs}$, where
\begin{equation}
  \mathop{\mathrm{v}\mathfrak{p}\mathrm{\_maxabs}}(\mathsf{x},\mathsf{y})=\mathop{\mathtt{\_mm?\_max\_pd}}(\mathop{\mathtt{\_mm?\_abs\_pd}}(\mathsf{x}),\mathop{\mathtt{\_mm?\_abs\_pd}}(\mathsf{y})),
  \label{e:vpI}
\end{equation}
and this operation is exact.  When $p=1$, the replacement for the
hypotenuse is
\begin{equation}
  \mathop{\mathrm{v}\mathfrak{p}\mathrm{\_sumabs}}(\mathsf{x},\mathsf{y})=\mathop{\mathtt{\_mm?\_add\_pd}}(\mathop{\mathtt{\_mm?\_abs\_pd}}(\mathsf{x}),\mathop{\mathtt{\_mm?\_abs\_pd}}(\mathsf{y})).
  \label{e:vp1}
\end{equation}
The absolute values in~\eqref{e:vpI} and~\eqref{e:vp1} are required
only at the lowest level of the recursion, since at the higher ones
all intermediate results are already non-negative.

It remains to generalize the $\mathtt{xNRMF}$ case ($p=2$) to any
other $p>0$.  Given real and finite $x$ and $y$, let
$\mathrm{M}=\max\{|x|,|y|\}$ and $\mathrm{m}=\min\{|x|,|y|\}$, and
observe that
\begin{equation}
  \|[x\,y]\|_p^{}=\!\sqrt[p]{|x|^p+|y|^p}=\mathrm{M}\sqrt[p]{1+q^p},\quad
  \mathrm{m}>0\!\implies\!q=\frac{\mathrm{m}}{\mathrm{M}},\
  \mathrm{m}=0\!\implies\!q=0.
  \label{e:sp}
\end{equation}
This is exactly how $\mathrm{v1\_hypot[f]}$ works when $p=2$.  For
$0<p<1$, $\|\mathbf{x}\|_p$ from~\eqref{e:p} is not a norm, but a
quasi- (or pre-)norm.  That case, although supported for not too small
$p$, is not in the focus here, but for its numerous applications see,
e.g., \cite{Bruckstein-et-al-09,Lai-Wang-11}.

In~\eqref{e:sp}, for $p\ge 1$ it holds
$1\le\sqrt[p]{1+q^p}\le\sqrt[p]{2}$, since $0\le q\le 1$, so
unwarranted overflow in an evaluation of~\eqref{e:sp} can occur only
due to rounding errors.  For $p\gg 1$ it is advisable not to compute
$\underline{q}^p$ directly, to avoid its underflow.  Instead, consider
\begin{equation}
  1+\underline{q}^p=1+(\underline{q}^{p/2})^2\approx\mathop{\mathrm{fma[f]}}(\underline{q^{p/2}},\underline{q^{p/2}},1).
  \label{e:qp}
\end{equation}

Let $\mathop{\mathrm{pow[f]}}(x,y)$ be any function that approximates
$x^y$ with a bounded relative error.  Its correctly rounded variant,
$\mathrm{cr\_pow[f]}$, is provided by the
CORE-MATH~\cite{Sibidanov-et-al-22} project.  Then, name the scalar
operation from Listing~5 that computes~\eqref{e:sp} as
$\mathrm{v1\_lp[f]}$, and substitute $\mathrm{v1\_lp[f]}$ for
$\mathrm{hypot[f]}$ in the scalar recursive algorithms $A$ and $B$.
For $A$ use $\mathrm{cr\_pow[f]}$ in $\mathrm{v1\_lp[f]}$, and
$\mathrm{\_\_builtin\_pow[f]}$ for $B$.  This completes the
generalization of $R_{\mathtt{x}}$ from the Frobenius to the $p$-norm,
but yet \emph{without any relative accuracy guarantees for a general}
$p\ge 1$ that would be similar to Theorem~\ref{thm:3}.

\begin{table}
\begin{minipage}{\textwidth}\centering\tablecaptionfont
  Listing~5: The $\mathrm{v1\_lp}$ operation in C\@.
\end{minipage}
\begin{lstlisting}
static inline double v1_lp(const double p, const double x, const double y) {
  const double X = __builtin_fabs(x); // $\text{X}=|\text{x}|$
  const double Y = __builtin_fabs(y); // $\text{Y}=|\text{y}|$
  const double m = __builtin_fmin(X, Y); // $\text{m}=\min\{\text{X},\text{Y}\}$
  const double M = __builtin_fmax(X, Y); // $\text{M}=\max\{\text{X},\text{Y}\}$
  const double q = (m / M); // might be a $\text{NaN}$ if, e.g., $\text{m}=\text{M}=0$, but$\ldots$
  const double Q = __builtin_fmax(q, 0.0); // $\ldots\text{Q}$ should not be a $\text{NaN}$
  const double S = pow(Q, (p * 0.5)); // $\text{S}\approx\text{Q}^{(\text{p}/2)}$
  const double Z = __builtin_fma(S, S, 1.0); // $\text{Z}\approx 1+\text{Q}^{\text{p}}$
  const double C = pow(Z, (1.0 / p)); // $\text{C}\approx\sqrt[\text{p}]{\text{Z}}$
  return (M * C); // $\text{M}\sqrt[\text{p}]{1+(\text{m}/\text{M})^{\text{p}}}\approx\sqrt[\text{p}]{|\text{x}|^{\text{p}}+|\text{y}|^{\text{p}}}$
} // if one argument of $\mathrm{fmin}$ or $\mathrm{fmax}$ is a $\text{NaN}$, the other argument is returned
\end{lstlisting}
\end{table}

Vectorization of $\mathrm{v1\_lp[f]}$ is straightforward, except for
$\mathrm{pow[f]}$.  With AVX-512F and the Intel's C/C++ compiler, the
intrinsics $\mathtt{\_mm512\_pow\_pd}$ and $\mathtt{\_mm512\_pow\_ps}$
are available.  Otherwise, the SLEEF~\cite{Shibata-Petrogalli-20}
functions $\mathtt{Sleef\_powd8\_u10avx512f}$ and
$\mathtt{Sleef\_powf16\_u10avx512f}$, with at most one ulp of error,
are recommended.  The other vector instruction subsets are similarly
covered.  This way $\mathrm{v}\mathfrak{p}\mathrm{\_lp[f]}$ is
obtained.

Replacing $\mathrm{v8\_hypot}$ with $\mathrm{v8\_lp}$ in Listing~4
completes the definition of the algorithm $Z_{\mathtt{D}}$ for the
vector $p$-norm computation.  The vectors containing the values of
$p/2$ and $\underline{1/p}$ should be defined once, at the highest
level of the recursion, instead of at each invocation of
$\mathrm{v}\mathfrak{p}\mathrm{\_lp}$, but that would probably require
a manual vector register assignment and a pure assembly implementation
of the whole algorithm.

If $\underline{q}=1$ in~\eqref{e:qp}, then, for a given
$\mathrm{pow[f]}$, there exist the smallest $p'>1$ such that
$\mathop{\mathrm{pow[f]}}(z,\underline{1/p'})=1$ for
$1\le z\le 2=1+\underline{q}^{p'\!/2}$, and thus
$\underline{\|[x\,y]\|_{p'}}=\mathrm{M}=\|[x\,y]\|_{\infty}$, due
to~\eqref{e:sp}, what agrees with
$\displaystyle\lim_{p\to\infty}\|\mathbf{x}\|_p=\|\mathbf{x}\|_{\infty}$.
This way the cutoff value of $p$, above which $\mathtt{xNRMP}$ should
switch to the faster code for $p=\infty$, can be determined.  A test
reveals that with $\mathrm{cr\_pow}$, $2^{52}<p'\le 2^{53}$, and with
$\mathrm{cr\_powf}$, $2^{23}<p'\le 2^{24}$.
%
%
\section{Numerical testing}\label{s:5}
%
%
The algorithms for the Frobenius norm computation were
tested\footnote{See
\url{https://github.com/venovako/VecNrmP/blob/master/testing.md} for
more setup details.} with GCC 14.2.1 on an Intel Xeon Cascadelake CPU,
running at 2.9\,GHz, while those for the vector $p$-norm were tested
with OpenCilk 3.0 on an Intel Xeon Phi 7210 CPU\@.  The timing
variability between the runs on the former system might be greater
than expected since its use was not exclusive, i.e., the machine load
was not predictable.

The testing setup is described in Section~\ref{s:2}.  Here, the timing
comparisons are shown first.  Let
$\mathop{\mathfrak{t}}(M_{\mathtt{x},t}^{\mathcal{D}})$ stand for the
wall time of the execution of $M_{\mathtt{x}}^{}$ in the run $t$ on
$\mathbf{x}_t^{}$ generated with the distribution $\mathcal{D}$ and
the seed $s_t^{\mathcal{D}}$.  Then, ``slowdown'' and ``speedup'' of
$M_{\mathtt{x},t}^{\mathcal{D}}$ versus
$N_{\mathtt{x},t}^{\mathcal{D}}$, are defined one reciprocally to the
other as
\begin{equation}
  \mathop{\mathrm{slowdown}_N^p}(M_{\mathtt{x},t}^{\mathcal{D}})=\mathop{\mathfrak{t}}(M_{\mathtt{x},t}^{\mathcal{D}})/\mathop{\mathfrak{t}}(N_{\mathtt{x},t}^{\mathcal{D}}),\quad
  \mathop{\mathrm{speedup}_N^p}(M_{\mathtt{x},t}^{\mathcal{D}})=\mathop{\mathfrak{t}}(N_{\mathtt{x},t}^{\mathcal{D}})/\mathop{\mathfrak{t}}(M_{\mathtt{x},t}^{\mathcal{D}}).
  \label{e:speed}
\end{equation}
In~\eqref{e:speed}, $N$ stands for a ``baseline'' algorithm, that
should be the most performant one, and $p$ denotes if the Frobenius
($p=2$) or another $p$-norm was computed.

Figure~\ref{fig:2} shows the slowdown of $M\in\{A,B,H\}$ versus $N=Z$
for $p=2$.  The scalar recursive algorithms are consistently slower
than $\mathtt{xNRMF}$, so only the vector ones should be compared
further, as done in Figure~\ref{fig:3}.  In no case $X$ was faster
than $Z$, and $Y$ only slightly, in a few cases.  This justifies the
vectorization with $\mathfrak{p}$ as large as possible, while
demonstrating that $Y$ is a reasonable fall-back for machines with
256-bit-wide vectors.  Therefore, $Z$ and $Y$ are the vectorized
recursive algorithm to be compared in the following with widely used
methods for computing the Frobenius norm, such as the $\mathtt{xNRM2}$
routines from the Reference LAPACK ($L$) and the Intel's sequential
Math Kernel Library, and $\mathtt{reproBLAS\_xnrm2}$ from ReproBLAS\@.

\begin{figure}
  \centering
  \includegraphics[keepaspectratio]{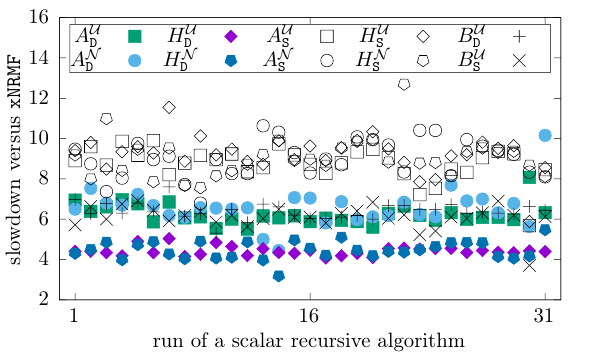}
  \caption{Slowdown~\eqref{e:speed} of the scalar recursive algorithms versus $\mathtt{xNRMF}$.}
  \label{fig:2}
\end{figure}

\begin{figure}
  \centering
  \includegraphics[keepaspectratio]{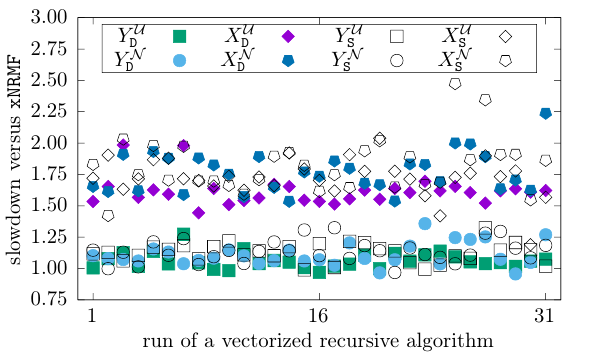}
  \caption{Slowdown~\eqref{e:speed} of the vectorized recursive algorithms versus $\mathtt{xNRMF}$.}
  \label{fig:3}
\end{figure}

Let the MKL's $\mathtt{xNRM2}$ routine be denoted by $J_{\mathtt{x}}$,
and the one from ReproBLAS by $K_{\mathtt{x}}$ (i.e.,
$K_{\mathtt{x}}\in\{\mathtt{reproBLAS\_dnrm2},\mathtt{reproBLAS\_snrm2}\}$).
In all measurements, $J$ was the fastest, albeit not publicly
specified method.  Thus, for Figure~\ref{fig:4}, $N=J$ has been taken.
It can be concluded that $\mathtt{xNRMF}$ is \emph{about twice slower}
than $K$, which in turn is somewhat slower than the MKL's method.
\emph{All three algorithms exhibited the relative accuracy of less
than two $\varepsilon_{\mathtt{x}}$ on the test inputs}.  The
strengths of $\mathtt{xNRMF}$ thus do not lie in its performance, but
in its simplicity, portability, and generalizability.

\begin{figure}
  \centering
  \includegraphics[keepaspectratio]{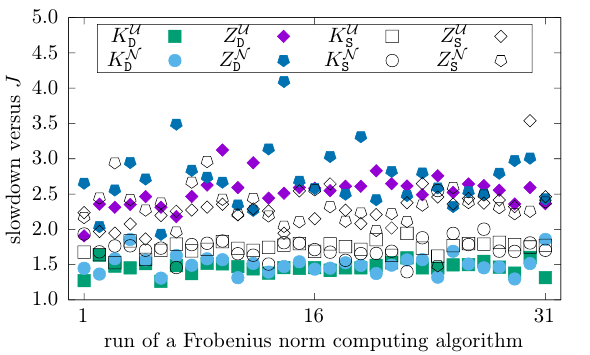}
  \caption{Slowdown~\eqref{e:speed} of ReproBLAS and $\mathtt{xNRMF}$ versus the MKL.}
  \label{fig:4}
\end{figure}

Yet, compared to $N=L$, the algorithms $Z$ and $Y$ are in many cases
faster, and only in a few somewhat slower, as shown in
Figure~\ref{fig:5}.  Therefore, wherever $L$ is used, $n$ is large,
and $J$ and $K$ are not available (i.e., mostly on non-Intel
architectures), a vectorized recursive algorithm is a viable, highly
relatively accurate replacement.

\begin{figure}
  \centering
  \includegraphics[keepaspectratio]{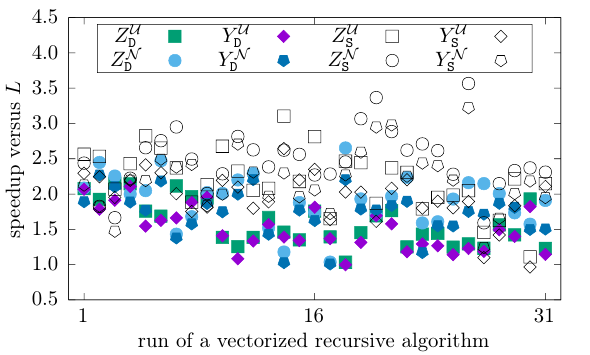}
  \caption{Speedup~\eqref{e:speed} of the most performant recursive vectorized algorithms versus $L$.}
  \label{fig:5}
\end{figure}

The OpenCilk parallelization is entirely optional.  It was tested
using input arrays of the length $n=2^{30}$, with
$\text{\texttt{CILK\_NWORKERS}}=2^k$, $0\le k\le 6$, worker threads.
The wall times of the parallel $Z_{\mathtt{D}}$ executions was
compared to the single-threaded ($k=0$) timing.  The speedup in each
run was consistently close to \texttt{CILK\_NWORKERS}.

For $p\ne 2$, the recursive algorithms are more of a prototype than an
optimized implementation, as described.  Therefore, only the maximal
relative error~\eqref{e:relerr} over all runs with a given $p$ is
shown in Table~4, also obtained with OpenCilk and $n=2^{30}$.

\begin{table}
  \begin{minipage}{\textwidth}\centering\tablecaptionfont
    Table~4: The maximal observed relative error~\eqref{e:relerr} of
    $A_{\mathtt{D}}$ and $Z_{\mathtt{D}}$ in multiples of
    $\varepsilon_{\mathtt{D}}$ for several $p$.
  \end{minipage}\\[1ex]
  \begin{tabular}{@{}cr@{}}\toprule
    $\approx p$ & $\mathop{\mathrm{relerr}}[A_{\mathtt{D}}]$\\\midrule
    $1/2$ & $2.958723$\\
    $2/3$ & $4.173733$\\
    $1$ & $1.253333$\\
    \bottomrule
  \end{tabular}\hfill
  \begin{tabular}{@{}cr@{}}\toprule
    $\approx p$ & $\mathop{\mathrm{relerr}}[A_{\mathtt{D}}]$\\\midrule
    $\sqrt{2}$ & $1.945139$\\
    $e$ & $3.161467$\\
    $\pi$ & $2.295947$\\
    \bottomrule
  \end{tabular}\hfill
  \begin{tabular}{@{}cr@{}}\toprule
    $\approx p$ & $\mathop{\mathrm{relerr}}[Z_{\mathtt{D}}]$\\\midrule
    $1/2$ & $3.374945$\\
    $2/3$ & $4.174019$\\
    $1$ & $1.253383$\\
    \bottomrule
  \end{tabular}\hfill
  \begin{tabular}{@{}cr@{}}\toprule
    $\approx p$ & $\mathop{\mathrm{relerr}}[Z_{\mathtt{D}}]$\\\midrule
    $\sqrt{2}$ & $3.890276$\\
    $e$ & $3.471359$\\
    $\pi$ & $3.222620$\\
    \bottomrule
  \end{tabular}
\end{table}
%
%
\section{Conclusions and future work}\label{s:6}
%
%
The Frobenius norm of an array $\mathbf{x}$ of the length $n$,
$\|\mathbf{x}\|_F$, might be computed with a significantly better
accuracy for large $n$ than with the Reference BLAS routine
$\mathtt{xNRM2}$, while staying in the same precision $\mathtt{x}$, by
using $\mathtt{xNRMF}$, a vectorized recursive algorithm proposed
here.  The performance of $\mathtt{xNRMF}$ should not differ much from
that of $\mathtt{xNRM2}$, but is lower than what the Intel's MKL and
ReproBLAS achieve.

Overflow avoidance of every intermediate result of $\mathtt{xNRMF}$ is
a purposefully built-in property of the algorithm.  All operations are
performed in the datatype of the input elements, and neither any
scaling nor elaborate accumulation schemes, as
in~\cite{Ahrens-et-al-20}, are required.  This simplicity comes with a
price in the terms of performance.

A more extensive testing is left for future work, where the magnitudes
of the elements of input arrays vary far more than in the tests
performed here.  It is also possible to construct an input array, or
sometimes permute a given one, that will favor $\mathtt{xNRM2}$ over
$\mathtt{xNRMF}$ in the terms of the result's accuracy, as hinted
throughout the paper.  Thus, it is important to bear in mind how both
algorithms work and choose the one better suited to the expected
structure and length of input arrays.  However, as $n$ increases, the
relative error of $\mathtt{xNRMF}$ grows at most logarithmically with
$n$, much slower than that of $\mathtt{xNRM2}$.  It is expected that
on a majority of large inputs $\mathtt{xNRMF}$ will exhibit a
noticeably lower error, at par with the MKL and ReproBLAS\@.

Unlike the routines from the closed-source MKL, an experienced user
can implement $\mathtt{xNRMF}$ on another vector architecture in a
day.  Of most interest would be those with the vector lengths beyond
$512$ bits, like, e.g., NEC SX-Aurora TSUBASA\footnote{See
\url{https://www.nec.com/en/global/solutions/hpc/sx/index.html}.}.
Portability of $\mathtt{xNRMP}$ depends on the availability of a
vectorized $\mathrm{pow[f]}$ function.

The unconditionally reproducible algorithm $A$ has already found an
application in a Jacobi-type method for the hyperbolic singular value
decomposition~\cite{Hari-Novakovic-26}.  Future work will focus on
improving the performance of $\mathtt{xNRMF}$ and $\mathtt{xNRMP}$, as
indicated throughout the paper.  The gains for the latter algorithm
might be significant.
%
%
\section*{Acknowledgements}
%
%
Some of the computing resources used have remained available to the
author after the project IP--2014--09--3670 ``Matrix Factorizations
and Block Diagonalization Algorithms''\footnote{See the MFBDA
project's web page at \url{https://web.math.pmf.unizg.hr/mfbda}.} by
Croatian Science Foundation expired.  The author would also like to
thank Dean Singer for his material support and declares no competing
interests.
%
%

%
%

\begin{thebibliography}{10}

\bibitem{Anderson-et-al-99}
E.~Anderson, Z.~Bai, C.~Bischof, S.~Blackford, J.~Demmel, J.~Dongarra,
  J.~Du~Croz, A.~Greenbaum, S.~Hammarling, A.~McKenney and D.~Sorensen, {\em
  {LAPACK} Users' Guide}.
\newblock Society for Industrial and Applied Mathematics, $3^{\rm rd}$ edn.,
  (1999).

\bibitem{Blue-78}
J.~L. Blue, A portable {F}ortran program to find the {E}uclidean norm of a
  vector, {\em {ACM} Trans. Math. Software} {\bf 4}(1)  (1978)  15--23.

\bibitem{Anderson-17}
E.~Anderson, Algorithm 978: Safe scaling in the {L}evel~1 {BLAS}, {\em {ACM}
  Trans. Math. Software} {\bf 44}(1)  (2017) art.~no.~12.

\bibitem{Novakovic-26}
V.~Novakovi{\'{c}}, Arithmetical enhancements of the {K}ogbetliantz method for
  the {SVD} of order two, {\em Numer. Algorithms} {\bf 101}(2)  (2026)
  1131--1156.

\bibitem{Sibidanov-et-al-22}
A.~Sibidanov, P.~Zimmermann and S.~Glondu, The {CORE}-{MATH} project, {\em
  29$\null^{th}$ IEEE Symposium on Computer Arithmetic (ARITH)}   (2022)
  26--34.

\bibitem{Borges-20}
C.~F. Borges, Algorithm 1014: An improved algorithm for hypot(x,y), {\em ACM
  Trans. Math. Softw.} {\bf 47}(1)  (2020) art.~no.~9.

\bibitem{Shibata-Petrogalli-20}
N.~Shibata and F.~Petrogalli, {SLEEF}: A portable vectorized library of {C}
  standard mathematical functions, {\em IEEE Trans. Parallel Distrib. Syst.}
  {\bf 31}(6)  (2020)  1316--1327.

\bibitem{Novakovic-23}
V.~Novakovi{\'{c}}, Vectorization of a thread-parallel {J}acobi singular value
  decomposition method, {\em {SIAM} J. Sci. Comput.} {\bf 45}(3)  (2023)
  C73--C100.

\bibitem{Schardl-Lee-23}
T.~B. Schardl and I.-T.~A. Lee, {OpenCilk}: A modular and extensible software
  infrastructure for fast task-parallel code, {\em 28$\null^{th}$ ACM SIGPLAN
  Annual Symposium on Principles and Practice of Parallel Programming (PPoPP)}
   (2023)  189--203.

\bibitem{Ahrens-et-al-20}
W.~Ahrens, J.~Demmel and H.~D. Nguyen, Algorithms for efficient reproducible
  floating point summation, {\em {ACM} Trans. Math. Software} {\bf 46}(3)
  (2020) art.~no.~22.

\bibitem{Fousse-et-al-07}
L.~Fousse, G.~Hanrot, V.~Lef\`{e}vre, P.~P\'{e}lissier and P.~Zimmermann,
  {MPFR}: A multiple-precision binary floating-point library with correct
  rounding, {\em ACM Trans. Math. Softw.} {\bf 33}(2)  (2007) art.~no.~13.

\bibitem{Bramas-17}
B.~Bramas, A novel hybrid quicksort algorithm vectorized using {AVX}-512 on
  {I}ntel {S}kylake, {\em Int. J. Adv. Comput. Sci. Appl.} {\bf 8}(10)  (2017)
  337--344.

\bibitem{OMPARB-24}
{OpenMP ARB}, {\em {OpenMP Application Programming Interface}}.
\newblock
  \url{https://www.openmp.org/wp-content/uploads/OpenMP-API-Specification-6-0.pdf},
   (2024).

\bibitem{Horn-Johnson-12}
R.~A. Horn and C.~R. Johnson, {\em Matrix Analysis}, $2^{\rm nd}$ edn.
  (Cambridge University Press, 2012).

\bibitem{Lefkimmiatis-et-al-13a}
S.~Lefkimmiatis, J.~P. Ward and M.~Unser, Hessian {S}chatten-norm
  regularization for linear inverse problems, {\em IEEE Trans. Image Process.}
  {\bf 22}(5)  (2013)  1873--1888.

\bibitem{Lefkimmiatis-et-al-13b}
S.~Lefkimmiatis and M.~Unser, Poisson image reconstruction with {H}essian
  {S}chatten-norm regularization, {\em IEEE Trans. Image Process.} {\bf 22}(11)
   (2013)  4314--4327.

\bibitem{Bruckstein-et-al-09}
A.~M. Bruckstein, D.~L. Donoho and M.~Elad, From sparse solutions of systems of
  equations to sparse modeling of signals and images, {\em SIAM Rev.} {\bf
  51}(1)  (2009)  34--81.

\bibitem{Lai-Wang-11}
M.-J. Lai and J.~Wang, An unconstrained $\ell_q$ minimization with $0<q\le 1$
  for sparse solution of underdetermined linear systems, {\em {SIAM} J. Optim.}
  {\bf 21}(1)  (2011)  82--101.

\bibitem{Hari-Novakovic-26}
V.~Hari and V.~Novakovi{\'{c}}, On convergence and accuracy of the
  {$J$}-{J}acobi method under the de\,{R}ijk pivot strategy, {\em
    Electron. Trans. Numer. Anal.} {\bf 65}  (2026)  26--62.

\end{thebibliography}
\end{document}